\documentclass[sn-mathphys,Numbered]{sn-jnl}

\usepackage{graphicx}%
\usepackage{multirow}%
\usepackage{amsmath,amssymb,amsfonts}%
\usepackage{amsthm}%
\usepackage{mathrsfs}%
\usepackage[title]{appendix}%
\usepackage{xcolor}%
\usepackage{textcomp}%
\usepackage{manyfoot}%
\usepackage{booktabs}%
\usepackage{algorithm}%
\usepackage{algorithmicx}%
\usepackage{algpseudocode}%
\usepackage{listings}%
\usepackage{mathtools}

\usepackage{tikz}
\usepackage{tikz-cd}
\usepackage{tikz-network}
\usepackage{comment}%
\usepackage{lmodern}
\usepackage[capitalize,nameinlink,compress]{cleveref} 

\newtheorem{theorem}{Theorem}
\newtheorem{example}{Example}%
\newtheorem{remark}{Remark}%

\newtheorem{definition}{Definition}%
\newtheorem{lemma}{Lemma}%
\newtheorem{assumption}{Assumption}

\DeclareMathOperator*{\argmin}{argmin}
\DeclareMathOperator*{\argmax}{argmax}
\newcommand{\R}{\mathbb{R}}
\newcommand{\dtr}{\mathrm{d}_{\mathrm{tr}}}
\newcommand{\drtr}{\mathrm{d}_{\gamma}}
\newcommand{\Rmax}{\mathbb{R}_{\max}}
\newcommand{\Rmin}{\mathbb{R}_{\min}}

\newcommand{\Tmin}{\mathbb{T}_{\min}}
\newcommand{\TPS}[1]{\mathbb{TP}^{#1}} 
\newcommand{\TPT}[1]{ \R^{#1}/\R \mathbf{1}}

\newcommand{\fw}{\mathrm{FW}} 
\newcommand{\tfw}{\mathrm{TFW}} 
\newcommand{\ones}{\mathbf{1}} 
\newcommand{\zeros}{\mathbf{0}} 

\newcommand{\jjs}[1]{\textcolor{red}{John: #1}}

\begin{document}

\title{Tropical Fermat--Weber Problems over Non-Finite Data and their Inverse Formulations}


\author*[1]{\fnm{John} \sur{Sabol}}\email{sabol.john.j@gmail.com}

\author[1]{\fnm{Ruriko} \sur{Yoshida}}\email{ruriko.yoshida@gmail.com}

\affil*[1]{\orgdiv{Department of Operations Research}, \orgname{Naval Postgraduate School}, \orgaddress{\street{1 University Circle}, \city{Monterey}, \postcode{93943}, \state{CA}, \country{USA}}}


\abstract{The term \emph{tropical pseudonorm} refers to a family of (not necessarily symmetric) gauge functions that arise in tropical or idempotent geometry. An important characteristic of these gauges is their invariance under translation by a constant vector, allowing them to descent naturally to tropical projective spaces. In this work, we explore the tropical one-infinity pseudonorm, a polyhedral hybrid gauge that allows for tunable asymmetry, in the context of a Fermat--Weber location problem. We extend previous formulations in considering non-finite data, and we investigate several variants of the inverse problem, providing linear programming formulations for their solution.}

\keywords{Location problem, Tropical metric, Hybrid norm, Inverse linear programming, Max-plus semiring, Polyhedral gauge functions}



\maketitle

\section{Introduction}\label{sec:intro}

Broadly speaking, a location problem can be described as follows: given a set $\mathcal{X}$, some finite subset $V\coloneq(v_1,\ldots,v_m)\subset \mathcal{X}$ of points (sometimes called \emph{customers} or \emph{sites}), and a map $d:\mathcal{X}\times \mathcal{X} \rightarrow \R$ measuring distance/cost between two points, find a point $x^*\in \mathcal{X}$ that minimize some function $f:\mathcal{X}\rightarrow \R$ of the distances between $x^*$ and the points in $V$. When the goal is to find a point $x^*$ that minimizes the weighted \emph{sum} of distances between itself and the points in $V$, i.e.,
\begin{equation}\label{eq:fw_definition}
    x^*\in \argmin_{x\in \mathcal{X}} f(x),\qquad f(x)=\sum_{i=1}^mw_id(x,v_i),
\end{equation}
the problem is often called the \emph{Fermat--Weber problem} (or sometimes simply the ``Weber problem''). In the classic setup, $w_i\in \R_{>0}$ are positive weights that denote the relative importance of the elements in $V$. Whenever weights are uniform, i.e., $w_i=w_j$ for all $i,j\in[m]\coloneq\{1,\ldots,m\}$, we say the problem is \emph{unweighted}. Unweighted Fermat--Weber problems of the form \cref{eq:fw_definition} are also called $1$-medians. The set of all such minimizers in $\mathcal{X}$ is called the \emph{Fermat--Weber set} of $V$, denoted as $\mathcal{X}^*$:
\[
\mathcal{X}^* \coloneq \bigl\{x^* : x^* \in \argmin_{x\in \mathcal{X}} \sum_{i=1}^m w_id(x,v_i\bigr\}.
\]
We will sometimes use notation like $\fw(\cdot)^*$ to denote the Fermat--Weber set of a specific problem instance. That is, we will use the asterisk as a notational device to distinguish between the Fermat--Weber \emph{problem} $\fw(\cdot)$, and its set of \emph{optimal solutions} $\fw(\cdot)^*$.

The Fermat--Weber problem is well-studied in the Euclidean setting, with Weiszfeld's algorithm \cite{Weiszfeld1937} providing an iterative method for computing a solution as early as 1937. Beyond the Euclidean setting, Fermat--Weber problems under $\ell^p$ norms, particularly the rectilinear ($\ell^1$) and Chebyshev ($\ell^\infty$) metrics in the plane, have also been studied extensively in the location-theory literature \cite{LoveMorrisWesolowsky1988, DreznerHamacher2002, IdrissiLoridanMichelot1988, HansenPeetersRichardThisse1985}. Many generalizations have been explored, such as for non-symmetric distance functions \cite{DurierMichelot1985}, and negatively weighted points \cite{Plastria1996}. Still, considerably less attention has been paid to projective spaces. 

Beginning with Lin and Yoshida \cite{Lin2016TropicalFP}, several authors have investigated the Fermat--Weber problem using a \emph{tropical distance}, a pseudometric (not necessarily symmetric) that is invariant under scalar addition, and which has found applications in the field of phylogenetics, see e.g., \cite{LinSturmfelsTangYoshida2023, Com_neci_2023, CoxCuriel2023}. The invariance under scalar addition suggests we work in a quotient space called the \emph{tropical projective torus}, denoted $\TPT{n}$. 

We continue work in this direction, using the hybrid tropical distance introduced in \cite{sabol2025tropicalfermatweberpolytropes} which generalizes the symmetric and asymmetric tropical distances considered previously. By leveraging optimality conditions associated with the subdifferential of the objective function, we extend the tropical Fermat--Weber problem to data that are not necessarily finite and show that finite solutions (if they exist) can be computed through straightforward modifications to the dual network flow graph. Then, we consider a class of tropical Fermat--Weber inverse problems which, in a certain sense, exchange the roles between certain input data and decision variables. We provide conditions on feasibility and explicit linear programming (LP) formulations for their solution.


\section{Preliminaries}\label{sec:preliminaries}

We begin with a brief overview of tropical mathematics and its intersection with location problems. $\langle \cdot,\cdot\rangle$ represents the Euclidean inner product, $\#$ the cardinality of a set, and $[n]=\{1,\ldots,n\}$ the set of all positive integers up to $n$. The non-negative (resp. positive) orthant is denoted $\R^n_{\geq0}$ (resp. $\R^n_{>0}$). We use $\zeros$ to denote the origin, and similarly, $\ones=(1,\ldots,1)$ for the vector of all-ones. The notation $\ones_K$ indicates the vector with $1$ in the $k$-th index for each $k\in K\subseteq [n]$ and zeros elsewhere. Thus, $\ones_k$ indicates the $k$-th standard basis vector.

\subsection{Tropical Mathematics}

Tropical mathematics refers to a branch of mathematics built on replacing the usual arithmetic operation of addition with either minimum or maximum, and replacing multiplication with addition. The choice in convention regarding \emph{tropical} addition results in the so-called \emph{min-plus} or \emph{max-plus} algebra. In the min-plus algebra $(\Rmin, \oplus, \odot)$, we have $\Rmin\coloneq \R \cup \{+\infty\}$ with $a\oplus b=\min(a,b)$ for $a,b\in \Rmin$, where $\infty$ serves as additive identity. Similarly, in the max-plus algebra $(\Rmax, \boxplus, \odot)$, we have $\Rmax \coloneq \R \cup\{-\infty\}$ with $a\boxplus b=\max(a,b)$ for $a,b\in \Rmax$ with $-\infty$ as additive identity. Note that the lack of additive inverses make these structures semirings. The isomorphism between min-plus and max-plus semirings is easily seen by the identity $-\min(-a,-b)=\max(a,b)$. Component-wise application of the operations leads to the tropical semimodule $\Rmin^n$ (resp. $\Rmax^n$). For $a\in \Rmin^n$, the subset of indices corresponding to finite elements define the \emph{support} of $a$, i.e. $\mathrm{supp}(a) \coloneq \{i \mid a_i\in \R\}$. 

Tropical geometry can be viewed as the piecewise-linear geometry induced by logarithmic coordinates on projective cones, where multiplicative structure becomes additive and asymptotic dominance relations become explicit. This perspective is especially useful in asymptotic analysis, optimization, and related areas such as large deviations and statistical mechanics. An important feature of tropical geometry is its invariance with regard to (tropical) scalar multiplication. That is, the elements of $\Rmin$ (resp. $\Rmax$) form an equivalence class given by the relation $x\sim x+c\ones$ for any $c\in\R$. We sometimes write $\tilde{x}$ to denote the equivalence class associated with $x$. This invariance implies that one operates in \emph{tropical projective space}, which for the min-plus semiring, is defined as
\[
\TPS{n-1}_{\mathrm{min}}\coloneq \bigl\{\Rmin \setminus \{\infty\ones\}\bigr\}/\R\ones.
\]
The \emph{tropical projective torus}, denoted $\TPT{n}$ is the subset of tropical projective space comprised of finite coordinates. Thus, whereas the definition of $\TPS{n-1}$ depends on the version of the tropical semiring, $\TPT{n}$ is the same under both conventions. When operating over $\TPT{n}$, it is common to a select a coordinate representative from the equivalence class according to some convention, such as by requiring that the first coordinate be equal to $0$. Thus, we see that $\TPT{n}$ naturally identifies with $\R^{n-1}$.

The natural metric over $\TPT{n}$ is given by the \emph{tropical distance},
\[
\dtr(x,y)\coloneq \max_{i\in[n]}(y_i-x_i)-\min_{i\in [n]}(y_i-x_i),
\]
which is closely related to (and sometimes called) \emph{Hilbert's projective metric}. $\dtr$ is induced by the \emph{tropical norm} $\lVert x \rVert_{\mathrm{tr}} \coloneq \max_i(x_i)-\min_i(x_i)$, which one can verify defines a norm over $\TPT{n}$.

\emph{Tropical hyperplanes} are fundamental objects in tropical geometry.
\begin{definition}[Min-Plus Tropical Hyperplane, \cite{Joswig2021Essentials}]\label{def:trop_hyperplanes}
    Given a vector $a\in \Rmin^n$ with $\mathrm{supp}(a)\geq 2$, the (min-plus) \emph{tropical hyperplane} is given by
    \[
    \mathcal{T}^{\min}(a)\coloneq \{x\in \R^n : \# \argmin_i(a_i + x_i) \geq 2 \}.
    \]
\end{definition}
An equivalent definition for the max-plus tropical hyperplane $\mathcal{T}^{\max}(a)$ for $a\in \Rmax^n$ is obtained by replacing $\argmin$ with $\argmax$ in the definition above. The point $-a$, called the \emph{apex}, witnesses every term $a_i+x_i$ attain the minimum (resp. maximum) simultaneously. Any tropical hyperplane with apex $-a$ can be viewed as a translation $\mathcal{T}(a)=(-a)+\mathcal{T}(\zeros)$. The union of several tropical hyperplanes $\bigcup_{a\in A}\mathcal{T}^{\min}(a)$ forms a tropical hyperplane arrangement.

A tropical hyperplane partitions $\TPT{n}$ into \emph{tropical sectors}, where each sector corresponds to the set of points for which the $\argmin$ (resp. $\argmax$) in \cref{def:trop_hyperplanes} is attained for a particular index.
\begin{definition}[Min-Plus Sectors, \cite{Joswig2021Essentials}]
    For $a\in \Rmin^n$, the $i$-th (closed) sector of the tropical  hyperplane $\mathcal{T}^{\min}(a)$ is the set
    \[
    S^{\min}_i(a) \coloneq \bigl\{x\in \R^n : a_i+x_i = \argmin_{j\in[n]}(a_j+x_j)\bigr\}.
    \]
\end{definition}
Open tropical sectors correspond to the set of points for which the $\argmin$ (resp. $\argmax$) is attained uniquely. Max-plus sectors are defined analogously.

\subsection{Tropical Gauges}

The following definitions are standard in convex analysis; see e.g., \cite{Rockafellar1970} or \cite{HiriartUrrutyLemarechal2001} for further background. 

A \emph{gauge} is a convex, nonnegative, and positively homogeneous function $\gamma:\R^n\rightarrow [0,\infty]$, where positive homogeneity means $\gamma(c x)=c \gamma(x)$ for all scalars $c \geq 0$. Equivalently, gauges may be generated from convex sets as follows.
\begin{definition}[Gauge of a set, \cite{Rockafellar1970}]\label{def:gauge}
    Let $B\subset \R^n$ be a convex set that contains the origin. The \emph{gauge} associated with $B$ is the function $\gamma:\R^n\rightarrow \Rmin$ defined as
    \[
    \gamma_B(x)\coloneq \inf \{c>0 : x\in c B\},
    \]
    where $c B=\{c x \mid x\in B\}$. The set $B$ plays the role of the ``unit ball'' of the gauge. 
\end{definition}

It is straightforward to show that $\gamma_B$ is convex and positively homogeneous, and thus a gauge. The \emph{polar set} of $B$ is defined as
\[
B^\circ \coloneq \{p : \langle p,x\rangle \le 1, \, \forall x\in B\}.
\]
The polar $B^\circ$ is convex, closed, bounded, and contains the origin in its interior. Its associated gauge $\gamma^\circ$ is called the \emph{dual gauge} of $\gamma$.

Any gauge induces a distance function $d_B$ between points of $\R^n$ as follows:
\[
\forall \,x,y \in \R^n, \quad d_B(x,y)\coloneq \gamma_B(y-x) = \inf\{c >0 : y-x\in c B\}.
\]
Convexity and positive homogeneity of a gauge imply subadditivity, so $d_B$ obeys the triangle inequality. 

A gauge is \emph{symmetric} if $\gamma(x)=\gamma(-x)$. It is \emph{definite} if $\gamma(x)=0$ \emph{only} when $x=\zeros$. Thus, if a symmetric and definite gauge takes on only finite values, then it defines a norm. Relaxing the requirement for symmetry in the norm axioms results in what some authors call an \emph{asymmetric norm}, or a \emph{pseudonorm}\footnote{The term ``pseudonorm'' (and also \emph{seminorm}) is sometimes used to refer to the relaxation of definiteness rather than symmetry, though the literature is mixed. We follow the convention of \cite{Luo2018} here.} depending on the literature. We work only over finite-dimensional spaces here, so that if $B$ is bounded with $\zeros \in \mathrm{int}(B)$, we get that $\gamma_B$ is definite and finite-valued.

A gauge defined as in \cref{def:gauge} is sometimes called the \emph{Minkowski functional} of $B$. It is a \emph{weak Minkowski norm} in the sense of \cite{PapadopoulosTroyanov2014}. Additionally, bbserve that for any two points $x,y\in \R^n$ and vector $v\in \R^n$,
\[
d_B(x+v,y+v)=\gamma_B\bigl((y+v)-(x+v)\bigr)=\gamma_B(y-x)=d_B(x,y).
\]
That is, $d_B$ is translation invariant.    

When $B$ is a polytope, i.e., $B=\mathrm{conv}(\{b_1,\ldots,b_k\})$, $d_B$ is said to be \emph{polyhedral}. Any polyhedral gauge can be written as a max of linear forms,
\[
\gamma(x) = \max_{i\in [m]} \langle a_i, x\rangle,
\]
from which it follows that $B^\circ=\mathrm{conv}(\{a_1,\ldots,a_m\})$. Thus, the dual gauge of a polyhedral gauge is also polyhedral. For each face $F$ of a convex polytope $\mathcal{P}\subset \R^n$, define its normal cone as
\[
\sigma(F)=\{a : \langle a, y-x \rangle \leq 0, y\in \mathcal{P}, x \in F\},
\]
i.e., the set of linear functionals whose $\argmax$ over $\mathcal{P}$ contains $F$. The set
\[
\mathcal{N}(\mathcal{P})=\{\sigma(F) : F\,\, \text{a face of}\,\,\mathcal{P}\},
\]
is the \emph{normal fan} of $\mathcal{P}$. Normal cones associated to vertices of $\mathcal{P}$ are full-dimensional. Thus, each \emph{maximal} cone $\sigma_i$ of $\mathcal{N}(\mathcal{P})$ can be written
\[
\sigma_i=\{x : \langle a_i,x \rangle \geq \langle a_j,x \rangle, \,j\in[m]\}=\{x : \langle a_i-a_j, x \rangle \geq 0, \,j\in [m]\}.
\]
If $\zeros \in \mathrm{int}(\mathcal{P})$, then $\mathcal{N}(\mathcal{P})$ can be constructed by taking the cone over each face of $\mathcal{P}^\circ$. Not that since $B^\circ$ contains $\zeros$ by definition, then for polyhedral $\gamma$ we have $B^{\circ\circ}=B$ so that $\mathcal{N}(B^\circ)$ is generated by the rays $\{b_1,\ldots,b_k\}$.

If $B$ is a simplex, then the induced gauge distance $d_B$ is called a \emph{simplicial distance}.
\begin{definition}[Simplicial Distance, \cite{AminiManjunath2010}]\label{def:simplicial_distance}
    Let $H_0\coloneq \bigl\{x\in \R^n : \sum_ix_i=0 \bigr\}$ denote the hyperplane whose coordinates sum to zero. For $x,y\in H_0$, the \emph{simplicial distance} from $x$ to $y$ is given by
    \[
    d_\triangle(x,y)\coloneq \inf \{c : y-x+ c\ones \geq \zeros\}=\max_i(x_i-y_i).
    \]
\end{definition}

The reason for working over $H_0$ comes from the fact that the standard simplex, $\triangle_{n-1}\coloneq\{x \in \R^n_{\geq0}: \sum_i x_i=1\}$ does not contain $\zeros$ and so does not directly define a gauge on $\R^n$. By utilizing a quotient map $\pi:\R^n\rightarrow \TPT{n}$, the authors in \cite{AminiManjunath2010} show that for $x\in H_0$, $x\in c \pi(\triangle_{n-1})$ if and only if there exists $t\in \R$ such that $x+t\ones\in c\triangle_{n-1}$. This is equivalent to requiring $x_i+t\geq0$ for $i\in[n]$ and $\sum_i(x_i+t)=c$. Since $x\in H_0$, we have $nt=c$, and so, $x_i+c/n \geq 0$. Equivalently, $c\geq -nx_i$ for $i\in[n]$. It follows that the associated gauge is 
\[
\gamma_\triangle(x)=\inf \{c >0 : x \in c\pi(\triangle_{n-1}) \}=n\max_i(-x_i).
\]
Consequently, we see that the simplicial distance is precisely the simplicial gauge scaled by a factor of $1/n$:
\[
d_\triangle(x,y)=\tfrac{1}{n}\gamma_\triangle(y-x).
\]


In \cite{Com_neci_2023}, the authors introduce the \emph{asymmetric tropical distance} function, which can be viewed as a generalization of $d_\triangle$ that allows one to work with points over $\TPT{n}$ (note that $H_0$ is just one particular convention for selecting representatives for equivalence classes in $\TPT{n}$).
\begin{definition}[Asymmetric Tropical Distance, \cite{Com_neci_2023}]
    For $x,y\in \TPT{n}$, the \emph{asymmetric tropical distance} from $x$ to $y$ is given by
    \begin{equation}\label{eq:asym_dtr}
    \dtr^\rightarrow(x,y)\coloneq n\max_i(x_i-y_i) - \sum_i(x_i-y_i).
    \end{equation}
    Its associated gauge is
    \[
    \gamma_\mathrm{tr}^\rightarrow(x)=n\max_i(-x_i)-\sum_i(-x_i)=\sum_i\bigl(x_i-\min_j(x_j)\bigr).
    \]
\end{definition}
If we select $\overline{x}\coloneq x-(\min_ix_i)$ as our choice of representative for $\tilde{x}$, i.e., the unique point intersecting the boundary of $\R^n_{\geq0}$ (called the \emph{canonical coordinate} of $x$), then we see that $\lambda_{\mathrm{tr}}^\rightarrow(\overline{x})=\lVert\overline{x}\rVert_1$ is equal to usual the $\ell^1$ norm. 

A function is said to be \emph{increasing} on an interval $I$ if
\[
x\leq y \implies f(x)\leq f(y) \qquad \text{for all }  x,y\in I.
\]
Applying the coordinate-wise partial order $x\leq y \iff x_i\leq y_i$ for all $i$, it can be shown that $f(x)=\lVert x\rVert_1$ is an increasing function on $\R^n_{\geq0}$. Moreover, for \emph{distinct} $x, y \in \R^n_{\geq0}$ with $x \leq y$, we get $f(x)<f(y)$, which means that $f(x)$ is a \emph{strictly} increasing function on $\R^n_{\geq0}$. Thus, $\gamma_{\mathrm{tr}}^\rightarrow(\overline{x})$ is a strictly increasing function under the coordinate-wise partial order.

The tropical gauge associated with the asymmetric tropical distance appears also in the work of Luo \cite{Luo2018} as a specific example of what they call ``$B^p$-pseudonorms.''\footnote{Note that Luo uses ``pseudo'' here to mean not necessarily symmetric.} For a given $1\leq p \leq \infty$, these pseudonorms can be written
\begin{align}
\gamma_p(x)=\begin{cases}
    \sqrt[\leftroot{-3}\uproot{3}p]{\sum_{i\in[n]}(x_i-\min_{j\in[n]}x_j)^p} \quad &\text{if $p\in [1,\infty)$},\\
    \max_{i\in[n]}x_i-\min_{j\in[n]}x_j & \text{if $p=\infty$}.
\end{cases}
\end{align}
Com{\u{a}}neci \cite{Comaneci2024Convexity} calls these ``\emph{tropical} $\ell^p$ norms,'' and shows that such ``norms'' have the property of being $\triangle$\emph{-star-quasiconvex}. 
\begin{definition}[$\triangle$-star-quasiconvex functions, \cite{Comaneci2024Convexity}]\label{def:star-quasiconvex}
    A function $f_v:\TPT{n}\rightarrow \R$ is called $\triangle$-\emph{star-quasiconvex} with kernel $v$ if $f_v(x)=\overline{\gamma}(x-v)$ for some increasing function $\gamma:\R^n_{\geq0}\rightarrow\R$, where $\overline{\gamma}(x)\coloneq \gamma(\overline{x})$ is the gauge applied to the canonical coordinate $\overline{x}$. If $\gamma$ is \emph{strictly} increasing, then $f$ is called \emph{strictly $\triangle$-star-quasiconvex}.
\end{definition}
The tropical $\ell^1$ norm, which corresponds to the asymmetric tropical gauge function $\gamma_\mathrm{tr}^{\rightarrow}$, is strictly $\triangle$-star quasiconvex. Conversely, the tropical $\ell^\infty$ norm, which induces the tropical metric over $\TPT{n}$, is $\triangle$-star-quasiconvex, but \emph{not} strictly $\triangle$-star-quasiconvex. Notice that \cref{def:star-quasiconvex} implicitly measures distances $v\rightarrow x$, which amounts to exchanging the order of the arguments in $\dtr^\rightarrow$. That is, if we use $\gamma_{\mathrm{tr}}^\rightarrow(x)$ as our gauge $\gamma$, then  
\[
f_v(x)\coloneq\overline{\gamma}(x-v)=\dtr^\rightarrow(x-v,\zeros)=\dtr^\rightarrow(v,x),
\]
which differs from our original convention in \cref{eq:fw_definition}.

\begin{remark}
    The authors in \cite{AllamigeonEtAl2018} use a distance function similar to the simplicial distance which they describe as the tropical analogue of the \emph{Funk metric} found in Hilbert's geometry \cite{PapadopoulosTroyanov2014}. Formally, for $x,y\in \Rmax$ define
    \[
    d_F(x,y)\coloneq \inf \{c \geq0 : x-y+c\ones\geq 0\}=\max\{0, \max_i(y_i-x_i)\},
    \]
     working under the convention $-\infty-(-\infty)=-\infty$. Comparing with $d_\triangle$, the apparent swap between $x$ and $y$ can be attributed to the fact that the authors of \cite{AminiManjunath2010} use the min-plus convention, whereas the authors of \cite{AllamigeonEtAl2018} use max-plus. 
\end{remark}

It is important to note that tropical $\ell^p$ norms are positive definite over $\TPT{n}$ (though not over $\R^n$), a statement which somewhat mirrors our earlier comment regarding the tropical norm.

\subsection{One-Infinity Tropical Pseudonorms}

In \cite{sabol2025tropicalfermatweberpolytropes}, the authors propose a hybrid tropical distance that in effect interpolates between the asymmetric tropical distance and the tropical metric. In what follows we focus exclusively on gauges associated with this hybrid tropical distance. Henceforth, with a slight abuse of notation, we use $\lambda$ as a parameter defining distance $\dtr^\lambda$, and use $\gamma_{\mathrm{tr}}^\lambda(x)$ to denote the associated gauge. Similarly, we will use $B_\lambda$ and $B^\circ_\lambda$ to denote the convex set and polar associated with a particular value of $\lambda$. 

\begin{definition}[Generalized Tropical Distance, \cite{sabol2025tropicalfermatweberpolytropes}]
    For $x,\,y\in\TPT{n}$ and $\lambda\in [0,1]$, the \emph{generalized tropical distance} from $x$ to $y$ is defined as
    \begin{align*}
    \dtr^\lambda(x,y)&\coloneq \lambda n\max_j(x_j-y_j)-(1-\lambda)n\min_j(x_j-y_j)+(1-2\lambda)\sum_j(x_j-y_j),\\
    &=\lambda \dtr^\rightarrow(x,y) + (1-\lambda) \dtr^\rightarrow(y,x).
    \end{align*}
    It has associated gauge $\gamma_\mathrm{tr}^\lambda(x) =\lambda \gamma_\mathrm{tr}^\rightarrow(x)+(1-\lambda)\gamma_\mathrm{tr}^\rightarrow(-x)$.
\end{definition}

$\gamma_\mathrm{tr}^\lambda(x)$ is a pseudonorm over $\TPT{n}$, which follows from the well-known fact that the properties of gauges/norms are retained under positively weighted sums.

\begin{lemma}\label{lem:star_quasiconvexity}
    A tropical gauge $\gamma_\mathrm{tr}^\lambda(x)$ is $\triangle$-star-quasiconvex. Furthermore, it is \emph{strictly} $\triangle$-star-quasiconvex for any $\lambda \neq 1/2$.
\end{lemma}
\begin{proof}
    For $\lambda\in(1/2,1]$, consider $x\neq y\in \R^n$ such that $\overline{x}\leq \overline{y}$ in canonical coordinates. By assumption, $\sum_i(\overline{y}_i-\overline{x}_i)>0$ and also $\max_i\overline{y}_i-\max_i\overline{x}_i\geq 0$. We see that
    \[
    \gamma_{\mathrm{tr}}^\lambda(\overline{y})-\gamma_{\mathrm{tr}}^\lambda(\overline{x}) = (1-\lambda)n\bigl(\max_i\overline{y}_i - \max_i\overline{x}_i\bigr)+(2\lambda-1)\sum_i(\overline{y}_i-\overline{x}_i)>0.
    \]
    For $\lambda\in[0,1/2)$, replace $\min$ with $\max$ in the definition of $\overline{x}$ and work over the non-positive orthant $\R^n_{\leq0}$ using the same proof. 
    To see that $\lambda=1/2$ fails to be strictly $\triangle$-star-quasiconvex, simply consider $x=(1,0,0)$ and $y=(1,1,0)$.
\end{proof}

We now introduce some notation that will be helpful. Let $\alpha= (1-\lambda)n$, $\beta= \lambda n$, and $\delta = (2\lambda-1)=(\beta-\alpha)/n$ so that 
\[
\gamma_{\mathrm{tr}}^\lambda(x)=\alpha\max_i(x_i)-\beta\min_i(x_i)+\delta\sum_ix_i=\max_{i,j}\langle \alpha \ones_i - \beta \ones_j + \delta \ones, x\rangle.
\]
Thus, $B^\circ_\lambda=\mathrm{conv}(\{a_{ij}\mid 1\leq i\neq j \leq n\})$ where $a_{ij}=\alpha \ones_i - \beta \ones_j + \delta \ones$. If we assume $\lambda \in (0,1)$ (so that $\alpha, \beta > 0$), then a maximal cone $\sigma^\circ_{ij}$ of $\mathcal{N}(B^\circ_\lambda)$ is
\begin{align}\label{eq:polar_cone}
    \sigma^\circ_{ij}&=\{x : \langle a_{ij}-a_{kl}, x\rangle\geq0, \, 1\leq k\neq l\leq n\}, \notag \\
    &=\{x : \langle \alpha (\ones_i-\ones_k)-\beta (\ones_j-\ones_l), x \rangle, \, 1\leq k\neq l \leq n\}, \notag\\
    &=\{x : \alpha x_i - \alpha x_k - \beta x_j + \beta x_l \geq 0, \, 1\leq k\neq l \leq n\},\\
    &=\{x : x_i \geq x_k, \, x_j \leq x_k, \, k\in [n]\}, \notag\\
    &=\{x : x\in S^{\max}_i(\zeros),\, x\in S^{\min}_j(\zeros)\}. \notag
\end{align}
Thus, each maximal cone of the normal fan to the polar dual is precisely the intersection of the $i$-th (closed) max-plus sector and the $j$-th (closed) min-plus sector where $i\neq j$. In other words, $\mathcal{N}(B^\circ_\lambda)$ is cut out by the union $\mathcal{T}^{\max}(\zeros) \cup \mathcal{T}^{\min}(\zeros)$. When $\lambda=0$, this simplifies to $\mathcal{N}(B^\circ_\lambda)=\mathcal{T}^{\max}(\zeros)$. Likewise, for $\lambda=1$, we obtain $\mathcal{N}(B^\circ_\lambda)=\mathcal{T}^{\min}(\zeros)$. To the best of our knowledge, these results for the asymmetric case first appeared in \cite{ComaneciPlastria2025Breakdown}. See \cref{fig:gauge_balls} for a depiction of the ``unit balls'' $B_\lambda$, associated polars $B^\circ_\lambda$, and the normal fans $\mathcal{N}(B^\circ_\lambda)$ for various values of $\lambda$ in $\TPT{3}$.

\begin{figure}
    \centering
    \includegraphics[width=1\linewidth]{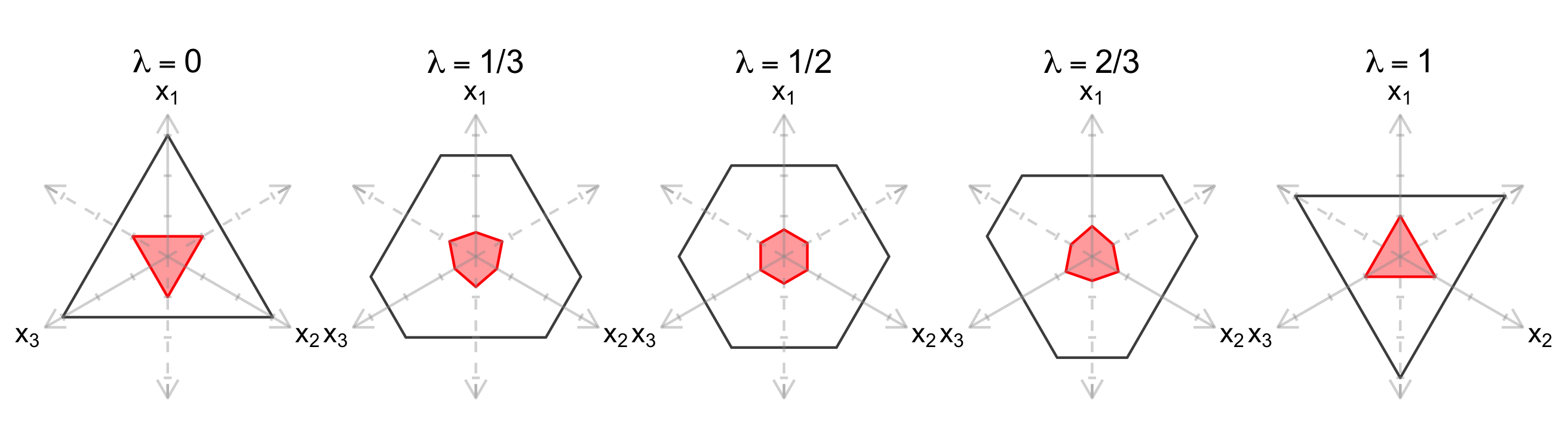}
    \caption{Polyhedral unit balls $B_\lambda$ (solid, red) for various $\lambda\in [0,1]$ along with the associated polar $B^\circ_\lambda$ (non-filled, black). Points are plotted on the hyperplane $H_0$ (per \cref{def:simplicial_distance}) using the ``Mercedes-Benz'' tight frame \cite{Waldron2018}. Positive (resp. negative) unit directions are shown with solid (resp. dashed) rays from the origin.}
    \label{fig:gauge_balls}
\end{figure}

Conceptually, $\gamma_{\mathrm{tr}}^\lambda$ can be viewed as the tropical analogue of the weighted \emph{one-infinity} hybrid norm $\ell^{1,\infty}$ introduced in \cite{WardWendell1980}, which interpolates between the rectilinear ($\ell^1$) and Chebyshev ($\ell^\infty$) norms. Here, we are interpolating between \emph{tropical} $\ell^1$ and the \emph{tropical} $\ell^\infty$ (pseudo)-norms, which explains our use of the term \emph{tropical $\ell^{1,\infty}$ norm}. Actually, we are interpolating between the $\underline{B}^1$-pseudonorm and the $\overline{B}^1$-pseudonorm (in the notation of \cite{Luo2018}) since we consider both tropical semirings simultaneously. $\lambda$ is the factor that controls the direction and degree of the asymmetry, i.e., the relative importance of distances \emph{to} vs. \emph{from} a particular point in the sample. Following \cite{WardWendell1980}, we call $\lambda$ the \emph{bias} of the gauge.

In analogy to the classical case, $B^p$-pseudonorms are polyhedral for $p\in\{1,\infty\}$, and thus, so too is $\gamma_{\mathrm{tr}}^\lambda$. Polyhedral norms/gauges allow one to formulate associated location problems as linear programs, and such formulations were provided \cite{Lin2016TropicalFP, Com_neci_2023, sabol2025tropicalfermatweberpolytropes} for tropical Fermat--Weber problems.

Consider a sample $A=\{a_1,\ldots,a_m\}\subset \R^n$, and for each $a_i\in A$, an associated tropical gauge parameterized by $\lambda_i\in[0,1]$. We use $\Lambda=(\lambda_i)_{i\in[m]}$ to denote the vector of biases that determine each tropical gauge, and use $a_i+\mathcal{N}_i$ as shorthand for the normal cone of the polar associated to each gauge translated by $a_i$. Taking the union of all such translated fans, we obtain a polyhedral decomposition of $\R^n$, which we denote by
\begin{equation}\label{eq:cellular_decomposition}
    \mathscr{C}(A,\Lambda)\coloneq \cup_{i}(a_i+\mathcal{N}_i).
\end{equation}

\begin{definition}[Elementary Convex Set, \cite{DurierMichelot1985}]\label{def:elementary}
    Let $p$ denote a vector of $B^\circ_\lambda$ such that $\gamma^\circ_\lambda(p)=1$ and have $\sigma(p)$ denote the corresponding normal cone in $\mathcal{N}(B^\circ_\lambda)$. Take one such $p_i$ for each $i\in [m]$, where $p_i$ can depend on $\lambda_i$. For a family $P\coloneq (p_i)_{i\in [m]}$, define the (possibly empty) set
    \[
    \mathcal{C}_P\coloneq \bigcap_{i\in [m]}(a_i+\sigma(p_i)).
    \]
    A nonempty convex set $C$ is called an \emph{elementary convex set} if there exists a family $P$ such that $C=\mathcal{C}_P$.
\end{definition}

\begin{remark}
    For the fully asymmetric case, $\Lambda=\zeros$ (resp. $\Lambda=\ones$), $\mathscr{C}(A,\Lambda)$ is the \emph{covector decomposition} generated by the tropical hyperplane arrangement $\cup_{a\in A}\mathcal{T}^{\max}(-a)$ (resp. $\cup_{a\in A}\mathcal{T}^{\min}(-a)$). For general $\Lambda$, $\mathscr{C}(A,\Lambda)$ was called the ``bi-tropical covector decomposition'' in \cite{sabol2025tropicalfermatweberpolytropes}. Each elementary convex set is a cell of $\mathscr{C}(A,\Lambda)$. Thus, the set of all elementary convex sets is indexed by the set of covectors in the associated covector decomposition. 
\end{remark}

\subsection{The Tropical Fermat--Weber Problem}

We now define our tropical Fermat--Weber problem. Given a sample of points $V=\{v_1,\ldots,v_m\}$, with weights $w=(w_1,\ldots,w_m)$, and biases $\Lambda=(\lambda_1,\ldots,\lambda_m)$ such that $v_i\in\R^n$, $w_i\in\R_{\geq 0}$, and $\lambda_i \in [0,1]$ for $i\in[m]$, the \emph{tropical Fermat--Weber problem} is
\begin{equation}\label{eq:TFW}
    \tfw(V,\Lambda,w): \quad \underset{x\in \TPT{n}}{\text{minimize}} \sum_{i=1}^m w_i\dtr^{\lambda_i}(x,v_i).
\end{equation}

We use $V_{ij}$ to denote the $j$-th coordinate of $v_i$. Additionally, we update our notation to incorporate weights and to index over the sample:
\[
\alpha_i^w = w_i(1-\lambda_i)n, \quad \beta_i^w = w_i\lambda_in, \quad \delta_i^w = w_i(2\lambda_i-1), \quad \text{and} \quad \delta^w_{tot}=\sum_i\delta^w_i.
\]
Using the epigraph formulation from linear programming, if we define primal variables $\phi_i,\,\psi_i$ such that 
\begin{equation}\label{eq:primal_constraints}
    \phi_i\geq x_j-V_{ij}, \qquad \psi_i \leq x_j-V_{ij}, \quad \text{for all } i\in[m],\, j\in[n],
\end{equation}
then $\tfw(V,\Lambda,w)$ can be solved via the following a linear program $(\mathrm{P})$:
\begin{align*}
    \mathrm{(P)} \quad \underset{\pi}{\text{maximize}} \quad & b^\top \pi \qquad \qquad & \mathrm{(D)} \quad \underset{\rho \geq 0}{\text{minimize}} \quad & c^\top \rho\\
    \text{s.t.}\quad & A^\top \pi \leq c, & \text{s.t.}\quad & A\rho = b.
\end{align*}
Here, $\pi$ denotes the collection of primal variables ($x,\phi,\psi$), $A^{\top}$ is a $(2mn)\times(n+2m)$ matrix whose constraints correspond to the inequalities in \cref{eq:primal_constraints}, and $b$ is a vector whose entries contain $\{\delta^w_{tot} \cup (-\beta_i^w)_{i\in[m]} \cup (\alpha_i^w)_{i\in [m]}\}$. The dual program $(\mathrm{D})$ is included on the right.

\begin{remark}\label{rem:total_unimodularity}
    From the formulation, it is immediate that matrix $A$ is totally unimodular and $(\mathrm{D})$ is a min-cost flow problem. Thus $(\mathrm{D})$ be solved in strongly polynomial time. Recovering optimal primal solutions for $(\mathrm{P})$ is accomplished by computing path closures over the residual graph, which is standard in network optimization. Details can be found in \cite{AMO_NetworkFlows} or \cite{sabol2025tropicalfermatweberpolytropes}. 
\end{remark}

As our tropical Fermat--Weber problem has linear costs, the following theorem applies.
\begin{theorem}[Durier--Michelot {\cite[Theorem 3.1]{DurierMichelot1985}}]\label{thm:Durier}
    Consider the Fermat--Weber problem with linear costs \cref{eq:fw_definition}. If $C$ is a elementary convex set associated with a family $P=(p_i)_{i\in [m]}$ (where $p_i$ is as in \cref{def:elementary}) such that $\sum_iw_ip_i=\zeros$, then $C$ is bounded, and $C=\mathcal{X}^*$.
\end{theorem}

Since the cone $\sigma(p_i)$ depends only on the gauge function $\gamma_i$, and since selecting $p_i$ doesn't depend in any way on $V$, evaluating $\sum_iw_ip_i$ can be done completely independent of the sample data $V$. Geometrically, the statement in the theorem corresponds to a selection of cones associated to vectors satisfying $\sum_iw_ip_i=\zeros$. If, upon translation of each cone $\sigma(p_i)$ by $v_i$, the resulting intersection $C=\cap_i(v_i+\sigma(p_i)$ is nonempty, then the by the theorem, $C=\mathcal{X}^*$ defines the optimal elementary convex set.

Since we have already shown that cones $\sigma(p_i)$ correspond to intersections of tropical sectors, and since the constraints of $(\mathrm{P})$ index the sectors of the tropical hyperplanes associated to each sample point, ``choosing'' a particular $p_i$ is essentially equivalent to selecting arcs (incident to the nodes associated with $\phi_i$ and $\psi_i$) in the min-cost flow graph that are permitted to carry non-zero flow.

\section{Samples with Non-Finite Coordinates}

Gauges are defined over a domain with finite coordinates. Tropical Fermat--Weber problems have been (to the best of our knowledge) likewise proposed only for samples $V\subset \TPT{n}$, rather than say, $V\subset \Rmax^n$ or $V\subset (\R\cup\{\pm\infty\})^n$. Indeed, direct attempts to solve (tropical) location problems for samples admitting non-finite coordinates appear fundamentally ill-posed. 

When the Fermat--Weber problem is equally weighted and given uniform, metric gauges, optimal solutions are often called \emph{geometric $1$-medians}, as they correspond to points that are ``central'' with regard to the sample data. Even in asymmetric cases, this notion of centrality pervades. For example, \cite{Com_neci_2023} calls optimal solutions to the unweighted, asymmetric Fermat--Weber problem ($\Lambda=\zeros$ or $\Lambda=\ones$) ``tropical medians'' of the sample, while \cite{JaggiKatzWagner2010} calls them ``tropical centerpoints.''

In the one-dimensional setting, medians of samples from $\R \cup \{\pm \infty\}$ are well-understood using order (vice distance) measures. In higher dimensional settings $n\geq2$, however, no standardized notion of median exists, and several possible definitions (e.g., coordinate-wise median, Oja's simplicial median \cite{Oja1983}, Tukey's half-space median \cite{Tukey1975}) are available. While some multi-dimensional definitions, such as those based on \emph{halfspace depth}, avoid distance-based measures, much of their literature still assumes finite data. 

In this section, we extend our previous linear programming framework for solving $\tfw$ to samples where the coordinates are not limited to finite elements. We begin with a small example to build intuition.

\begin{example}\label{ex:running_example}
    Consider the symmetric, unweighted problem for $V\subset \Rmax^n$ consisting of three parameterized points: $v_1=(a,2,2)^\top$, $v_2=(0,b,1)^\top$, and $v_3=(0,2,c)^\top$ for $a,b,c\in\Rmax$. \cref{fig:extending_example} depicts the setup, where we plot all points under the projection that sets the first coordinate to zero. One may verify that for any finite, non-positive values of $a,b,c$, the tropical Fermat--Weber set $\mathcal{X}^*$ does not change. Even when $a,b,c=-\infty$, a \emph{min-plus} tropical hyperplane with apex at any point in $V$ remains well-defined. The rightmost figure depicts the covector decomposition of $\TPT{3}$ in the compactification $\TPS{2}$, see \cite[Chapter 6]{Joswig2021Essentials}, which also induces a covector decomposition of the \emph{boundary stratum} shown by the outmost circle.
\end{example}

\begin{figure}
    \centering
    \begin{tikzpicture}%
    \fill[{gray!20}](1,1) -- (2,2) -- (2,1);
    \draw [-,](2,2) -- (2,3.3);
    \draw [-,](2,2) -- (3,2);
    \draw [-,](2,2) -- (-0.5,-0.5);
    \draw [-,dashed](2,2) -- (-0.8,2);
    \draw [->,dashed](2,2) -- (2.8,2.8);
    
    \draw [-,](0,1) -- (0,2.5);
    \draw [-,](0,1) -- (3,1);
    \draw [-,](0,1) -- (-0.8,0.2);
    \draw [-,dashed](0,1) -- (0,-0.8);
    \draw [->,dashed](0,1) -- (-0.8,1);
    \draw [-,dashed](0,1) -- (2.2,3.2);
    
    \draw [-,](2,0) -- (3,0);
    \draw [-,](2,0) -- (1.2,-0.8);
    \draw [->,dashed](2,0) -- (2,-0.8);
    \draw [-,dashed](2,0) -- (-0.5,0);
    \draw [-,dashed](2,0) -- (2.8,0.8);
    \draw [-,dash dot](2,0) -- (2,2);
    
    \node at (2,2)[circle,draw=black,fill=white,inner sep=1.5pt]{};
    \node at (0,1)[circle,draw=black,fill=white,inner sep=1.5pt]{};
    \node at (2,0)[circle,draw=black,fill=white,inner sep=1.5pt]{};
    \node at (2.3,1.8) {\small$v_1$};
    \node at (-0.2,1.2) {\small$v_2$};
    \node at (2.3,-0.2) {\small$v_3$};
    \node at (1.7,1.3) {\small $\mathcal{X}^*$};
    \begin{scope}[xshift=5cm,yshift=0.4cm]
        \fill[{gray!20}](0.5,0.8) -- (1.2,1.5) -- (1.2,0.8);
        \draw [dotted] (1,1) circle (1.8cm);
        \draw [->](2.1,2.4) -- (-0.4,-0.1);
        \draw [<-](2.75,0.8) -- (-0.75,0.8);
        \draw [<-](1.2,2.75) -- (1.2,-0.75);
        \node at (2.11,2.41)[circle,draw=black,fill=white,inner sep=1.5pt]{};
        \node at (-0.78,0.8)[circle,draw=black,fill=white,inner sep=1.5pt]{};
        \node at (1.2,-0.78)[circle,draw=black,fill=white,inner sep=1.5pt]{};
        \node at (2.4,2.55) {\small$v_1$};
        \node at (-1.1,0.8) {\small$v_2$};
        \node at (1.2,-1.1) {\small$v_3$};
        \node at (1,1) {\tiny$\mathcal{X}^*$};
    \end{scope}
    \end{tikzpicture}
    \caption{Left: A sample (open circles) $V=\{v_1,v_2,v_3\}$ in $\Rmax^3$ plotted on the plane $x_1=0$ with $a=b=c=0$. The Fermat--Weber set, $\mathcal{X}^*$ is shown by the shaded triangle. Min-plus (solid) and max-plus (dashed) tropical hyperplanes with apex at each sample point depict the cellular decomposition $\mathscr{C}(V,\Lambda)$. Right: The same problem for $a=b=c=-\infty$, where the large circle represents the boundary strata. Min-plus hyperplanes, which remain well-defined for apices at these points, are retained.}
    \label{fig:extending_example}
\end{figure}
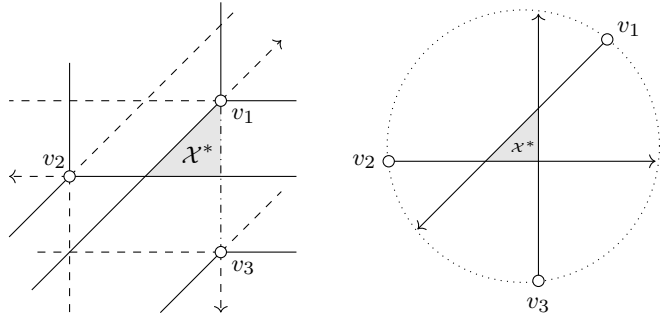

The previous example shows that from a combinatorial perspective, nothing has changed regarding the max-plus covector of the cell associated with $x^*$. In fact, this would also be true for the min-plus covectors except for the fact that min-plus tropical hyperplanes are not defined over elements of $\Rmax$. Still, the intuition is the right one. 
\begin{definition}\label{def:subdiff_of_gauge}
    The \emph{subdifferential} of gauge $\gamma$ at $x$ is
    \[
    \partial \gamma (x) = \{p\in B^\circ : \langle p,x\rangle=\gamma(x)\}.
    \]
\end{definition}

Finding a family $P=(p_i)_{i\in[m]}$ such that $\sum_iw_ip_i=0$ as in \cref{thm:Durier} has the interpretation of selecting a subgradient for each point in the sample such that their weighted sum satisfies the first-order optimality condition $\zeros\in \partial f(x^*)$ of an unconstrained convex optimization problem. It is useful then to note that only the direction (and not the distance) from $x\in \R^n$ to a point $v_i\in V$ is necessary in defining the subdifferential at $x$. That is, all that matters is which cones of $\mathcal{N}(B^\circ_\lambda)$ with vertex $x$ contain $v_i$.

Consider $V=\{v_1,\ldots,v_m\}$, where for each point $v_i\in (\R\cup \{\pm \infty\})^n$, $i\in [m]$, we make the following assumption.

\begin{assumption}\label{assumption:not_identically_infinite}
    $v_i$ is not identically $\mathbf{\infty}$ or $\mathbf{-\infty}$. That is, $v_i\neq (\pm\infty)\ones$.
\end{assumption}
Assume also that each point is assigned some tropical gauge $\gamma^{\lambda_i}_{\mathrm{tr}}$. The next lemma shows how we can define the subdifferential at $x\in \R^n$ in this setup.
\begin{lemma}\label{lem:subdiff_defined}
    For $f(x)=\sum_iw_i\gamma^{\lambda_i}_{\mathrm{tr}}(v_i-x)$ where each $\gamma^{\lambda_i}_{\mathrm{tr}}$ is the tropical gauge function associated with $\lambda_i$ and each $v_i\in (\R\cup \{\pm \infty\})^n$ satisfies \cref{assumption:not_identically_infinite}, the subdifferential $\partial f(x)$  evaluated at $x\in \R^n$ is well-defined and equals
    \[
    \partial f(x)=\sum_i w_i \mathrm{conv}\bigl(\{\alpha_i\ones_j-\beta_i\ones_{j'} + \delta_i\ones : j\in M_i, j'\in N_i\}\bigr),
    \]
    where $\alpha_i=(1-\lambda_i)n$, $\beta_i=\lambda_i n$, $\delta_i=(\beta_i-\alpha_i)/n$, and
    \[
    M_i\coloneq \argmax_j(V_{ij}-x_j), \quad N_i\coloneq \argmin_j(V_{ij}-x_j).
    \]
\end{lemma}
\begin{proof}
    Take arbitrary $i\in[m]$ and define a sequence $u^{(k)}\in \R^n$, where $u^{(1)}=x$, $u^{(k)}\rightarrow v_i$ coordinate-wise, and so that $u^{(k)}\in x+\sigma^\circ$ $\forall k$ and for some normal cone $\sigma^\circ$ of the polar dual of the gauge associated with $v_i$. Using \cref{def:subdiff_of_gauge} and the fact that the gauge is polyhedral, for any $u^{(k)}$ we can write 
    \[
    \partial f_i^{(k)}(x) = \partial \gamma_{\mathrm{tr}}^{\lambda_i}(u^{(k)}-x) = \mathrm{conv}(\{\alpha_i\ones_j-\beta_i\ones_{j'} +\delta_i\ones : j\in M_i^{(k)}, j'\in N_i^{(k)}\}),
    \]
    where 
    \[
    M_i^{(k)}\coloneq \argmax_j(u^{(k)}_j-x_j), \quad N_i^{(k)}\coloneq \argmin_j(u^{(k)}_j-x_j).
    \]
    Our construction of the sequence implies that the sets $M_i^{(k)}$ and $N_i^{(k)}$ never change for any $k$. By considering $k$ in the limit, we conclude
    \[
    \partial f_i(x)=\lim_{k\rightarrow \infty}\partial f_i^{(k)}(x)=\mathrm{conv}\{\alpha_i\ones_j-\beta_i\ones_{j'} + \delta_i\ones \mid j\in M_i^{(k)}, j'\in N_i^{(k)}\}.
    \]
    Since $i$ was arbitrary, we can do this for every point and take $\partial f(x)=\sum_i w_i \partial f_i(x)$, which is well-defined.
\end{proof}
Note that \cref{assumption:not_identically_infinite} is necessary to ensure sets $M_i$ and $N_i$ remain disjoint, which in turn guarantees the existence of some $\sigma^{\circ}$ satisfying the sequence $u^{(k)}$. In other words, the subdifferential of a tropical distance admits a canonical extension to points $V \subset (\R\cup \{\pm \infty\})^n$ so long as no points are identically $\mathbf{\infty}$ or $\mathbf{-\infty}$. We now have the following equivalent characterization of a tropical Fermat--Weber set.
\begin{definition}\label{def:fw_subdifferential}
    Let $\partial f(x)$ denote the tropical Fermat--Weber subdifferential evaluated at $x\in \R^n$, i.e., $f(x)=\sum_iw_i\gamma_{i}(v_i-x)$ for sample points $v_i\in (\R\cup \{\pm \infty\})^n$ satisfying \cref{assumption:not_identically_infinite}, $w_i>0$, and tropical gauge functions $\gamma_{\mathrm{tr}}^{\lambda_i}$. The \emph{tropical Fermat--Weber set} $\mathcal{X^*}$ is
    \[
    \mathcal{X^*}\coloneq \bigl\{x\in \TPT{n} : \zeros \in \partial f(x) \bigr\} .
    \]
\end{definition}
\cref{def:fw_subdifferential} offers the possibility of applying a subgradient-based method for finding an optimal $x^*\in \mathcal{X^*}$. This approach was proposed in \cite{sabol2025tropicalfermatweberpolytropes}, and we see that it extends quite naturally to the case of non-finite data. Nonetheless, we will pursue a strategy here based on network flows.

Select arbitrary $v_i\in V$ along with its gauge. We assume $\lambda_i \in (0,1)$, as the other cases follow from straightforward simplifications. Set 
\[
M^{\infty}_i \coloneq \{j : V_{ij}=+\infty\}\subsetneq [n], \quad N^{\infty}_i \coloneq \{j : V_{ij}=-\infty\}\subsetneq [n],
\]
either of which may be empty. For arbitrary $x\in\R^n$, if $M^{\infty}_i\neq\emptyset$, it must be the case that
\[
\argmin_j(x_j-V_{ij})=M^{\infty}_i, \qquad \text{and} \qquad \{\argmax_j(x_j-V_{ij})\cap M^{\infty}_i\}=\emptyset.
\]
The second statement follows from \cref{assumption:not_identically_infinite}. Consider a maximal cone $\sigma^\circ(p_i)$, where $p_i$ is an arbitrary vertex of $B^\circ_{\lambda_i}$, and recall that $\sigma^\circ(p_i)=S^{\max}_k(\zeros)\cap S^{\min}_l(\zeros)$ for some $k\neq l\in[n]$. It is obvious that if $k\in M^{\infty}_i$, then $\mathcal{C}_P=\emptyset$ must be empty for any family $P$ containing $p_i$. Otherwise, this would imply $k\in\argmax_j(x_j-V_{ij})$ for all $x\in \mathcal{C}_P$, a contradiction. Conversely, $\mathcal{C}_P$ can only be nonempty if $l \in M^{\infty}_i$ because otherwise it would imply $x_{l'}-V_{il'}<x_l-V_{il}=-\infty$ for some $l'\neq l$, which is impossible. If $N^{\infty}_i\neq \emptyset$, then symmetric arguments prohibit $l\in N^{\infty}_i$ for force $k \in N^{\infty}_i$.

Intuitively, our previous arguments correspond to the obvious fact that in a min-cost flow graph with arcs costs permitted to take on values $\pm\infty$, certain arcs are ``prohibited'' in any optimal flow. Equivalently, one could consider the corresponding primal constraints that are tight everywhere/nowhere for $x\in\R^n$ when allowing $\phi_i,\,\psi_i$ to take values in $\Rmax$.

Let $\mathcal{G}$ refer to the graph corresponding to the min-cost flow problem $(\mathrm{D})$ and $\mathcal{G}^-$ denote the subgraph of $\mathcal{G}$ that removes all arcs whose corresponding constraints are tight nowhere for $x\in \R^n$ (see \cref{fig:mcf_formulation}). Explicitly, $\mathcal{G}^-$ has arc set $\mathcal{A}^- = \mathcal{A}^-_{\max}\cup \mathcal{A}^-_{\min}$, where

\begin{equation}\label{eq:allowable_arcs}
    \begin{aligned}
    \mathcal{A}^-_{\max}
    &=
    \{(i,j) : V_{ij}=-\infty\}
    \cup
    \{(i,j) : \min_j V_{ij} > -\infty,\; V_{ij}<+\infty\},
    \\
    \mathcal{A}^-_{\min}
    &=
    \{(j,i') : V_{ij}=+\infty\}
    \cup
    \{(j,i') : \max_j V_{ij} < +\infty,\; V_{ij}>- \infty\}.
    \end{aligned}
\end{equation}

\begin{figure}
    \centering
    \centering
    \begin{tikzpicture}[node distance={30mm}, thick, main/.style = {draw, circle}]
        \draw [solid](0,2) to (1.5,2);
        \draw [draw=gray,dotted](0,2) to (1.5,1);
        \draw [draw=gray,dotted](0,2) to (1.5,0);
        \draw [draw=gray,dotted](0,1) to (1.5,2);
        \draw [solid](0,1) to (1.5,1);
        \draw [draw=gray,dotted](0,1) to (1.5,0);
        \draw [draw=gray,dotted](0,0) to (1.5,2);
        \draw [draw=gray,dotted](0,0) to (1.5,1);
        \draw [solid](0,0) to (1.5,0);

        \draw [draw=gray,dotted](3,2) to (1.5,2);
        \draw [solid](3,2) to (1.5,1);
        \draw [solid](3,2) to (1.5,0);
        \draw [solid](3,1) to (1.5,2);
        \draw [draw=gray,dotted](3,1) to (1.5,1);
        \draw [solid](3,1) to (1.5,0);
        \draw [solid](3,0) to (1.5,2);
        \draw [solid](3,0) to (1.5,1);
        \draw [draw=gray,dotted](3,0) to (1.5,0);

        \node at (0, 0)[circle,fill,inner sep=1.5pt]{};
        \node at (0, 1)[circle,fill,inner sep=1.5pt]{};
        \node at (0, 2)[circle,fill,inner sep=1.5pt]{};
        
        \node at (1.5, 0)[circle,draw=black,fill=white,inner sep=1.5pt]{};
        \node at (1.5, 1)[circle,draw=black,fill=white,inner sep=1.5pt]{};
        \node at (1.5, 2)[circle,draw=black,fill=white,inner sep=1.5pt]{};
        
        \node at (3, 0)[circle,fill,inner sep=1.5pt]{};
        \node at (3, 1)[circle,fill,inner sep=1.5pt]{};
        \node at (3, 2)[circle,fill,inner sep=1.5pt]{};
        
        \node[] at (-0.4, 0) {$\phi_3$};
        \node[] at (-0.4, 1) {$\phi_2$};
        \node[] at (-0.4, 2) {$\phi_1$};
        \node[] at (3.4, 0) {$\psi_3$};
        \node[] at (3.4, 1) {$\psi_2$};
        \node[] at (3.4, 2) {$\psi_1$};
        \node[] at (1.5, 2.5) {$\mathcal{G}^-$};
        \node[] at (0.75, -0.3) {$\mathcal{A}^-_{\max}$};
        \node[] at (2.25, -0.3) {$\mathcal{A}^-_{\min}$};
        \begin{scope}[xshift=6cm, yshift=0cm]
            \draw [gray, very thin, dashed](0,2) to [bend left=15](1.5,2);
            \draw [gray, very thin, dashed](0,1) to [bend left=15](1.5,1);
            \draw [gray, very thin, dashed](0,0) to [bend left=15](1.5,0);
            \draw [solid](0,2) to [bend right=15](1.5,2);
            \draw [solid](0,1) to [bend right=15](1.5,1);
            \draw [solid](0,0) to [bend right=15](1.5,0);
            
            \draw [solid](3,2) to [bend left=10](1.5,0);
            \draw [gray, very thin, dashed](3,2) to [bend right=10](1.5,0);
            \draw [solid](3,2) to [bend left=0](1.5,1);
            
            \draw [solid](3,1) to [bend left=10](1.5,2);
            \draw [gray, very thin, dashed](3,1) to [bend right=10](1.5,2);
            \draw [solid](3,1) to [bend left=0](1.5,0);
            
            \draw [solid](3,0) to [bend left=10] (1.5,1);
            \draw [gray, very thin, dashed](3,0) to [bend right=10] (1.5,1);
            \draw [solid](3,0) to [bend left=0] (1.5,2);
    
            \node at (0, 0)[circle,fill,inner sep=1.5pt]{};
            \node at (0, 1)[circle,fill,inner sep=1.5pt]{};
            \node at (0, 2)[circle,fill,inner sep=1.5pt]{};
            
            \node at (1.5, 0)[circle,draw=black,fill=white,inner sep=1.5pt]{};
            \node at (1.5, 1)[circle,draw=black,fill=white,inner sep=1.5pt]{};
            \node at (1.5, 2)[circle,draw=black,fill=white,inner sep=1.5pt]{};
            
            \node at (3, 0)[circle,fill,inner sep=1.5pt]{};
            \node at (3, 1)[circle,fill,inner sep=1.5pt]{};
            \node at (3, 2)[circle,fill,inner sep=1.5pt]{};
            
            \node[] at (-0.4, 0) {$\phi_3$};
            \node[] at (-0.4, 1) {$\phi_2$};
            \node[] at (-0.4, 2) {$\phi_1$};
            \node[] at (3.4, 0) {$\psi_3$};
            \node[] at (3.4, 1) {$\psi_2$};
            \node[] at (3.4, 2) {$\psi_1$};
            \node[] at (1.5, 2.5) {$\mathcal{G}^-_{\mathrm{res}}(\rho^*)$};
        \end{scope}
    \end{tikzpicture}
    \caption{On the left, the subgraph $\mathcal{G}^-$ of permissible arcs for $V$ and $\Lambda$ is in \cref{ex:running_example}. Dotted edges in the graph correspond to arcs that cannot carry positive flow. On the right, the residual graph corresponding to an optimal flow solution $\rho^*$. The dashed arcs represent the backwards arcs that are added to form the residual graph.}
    \label{fig:mcf_formulation}
\end{figure}
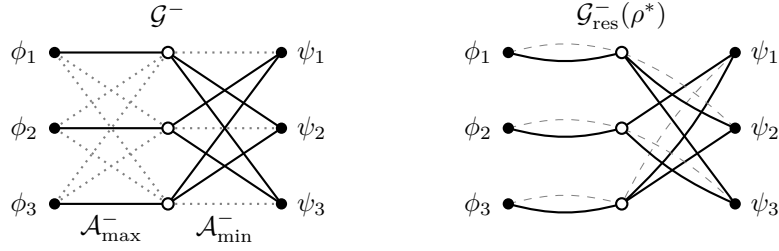

Note that while the graph $\mathcal{G}^{-}$ does not contain any arcs with weight $+\infty$, it still may contain arcs with weight $-\infty$. Thus, before we can solve $(\mathrm{MCF})$ over $\mathcal{G}^{-}$, we need to reassign real weights to any such arcs. The next lemma shows that this assignment may be arbitrary, and won't affect $\mathcal{X}^*$.

\begin{lemma}\label{lem:arbitrary_a}
    Let $c\in\R$ denote the reassigned weight of every arc in $\mathcal{G}^{-}$ whose original weight was $-\infty$, and let $\rho^*$ denote an arbitrary optimal flow of the problem $\mathrm{MCF}$ over $\mathcal{G}^{-}$. Then $D^*\in \Tmin^{n \times n}$, the Kleene star matrix associated with the $n$ ``$x$-nodes,'' i.e., the variables $x_1,\ldots,x_n$ in $(\mathrm{P})$, of $\mathcal{G}^{-}_{\mathrm{res}}(\rho^*)$ does not depend on $c$.
\end{lemma}
\begin{proof}
    Clearly, the value of $c$ does not impact the feasibility of $\mathrm{MCF}$. Thus, assume $\mathrm{MCF}$ over $\mathcal{G}^{-}$ is feasible as otherwise there is nothing to prove. Select arbitrary indices $s,t\in[n]$. We want to find the shortest-path distance between $x_{s}\rightarrow x_{t}$ in $\mathcal{G}^-_{\mathrm{res}}(\rho^*)$. Observe that any path must have even length since $\mathcal{G}^-$ is bipartite. Let $e$ be any arc leaving node $x_{s}$, and $w_e$ its associated weight. Write $e=(s,i)$ if $e$ is a backwards arc (i.e., $\rho^*_e>0$) and $e=(s,k)$ otherwise. Suppose, $e=(s,i)$. There are two possibilities: $V_{is}=-\infty$ (which implies $w_e=-c$ due to the negation of edge weights when constructing the residual graph), or $V_{is}\in\R$. Consider $e'$, the next edge selected in the path. If $V_{is}=-\infty$, then we must have $w_{e'}=c$ based on how we constructed $\mathcal{G}^-$. Thus, $w_e+w_{e'}=-c+c=0$. Conversely, if $V_{is}\in \R$, then $w_{e'}$ is independent of $c$ because $e'$ did not have its weight reassigned when constructing $\mathcal{G}^-$. The same arguments apply for $e=(s,k)$ if we select a forward arc initially. Since any optimal path $x_s\rightarrow x_t$ must have finite length, we can apply the previous arguments inductively and conclude that in any such path, the total path weight contains no $c$ terms due to cancellation.
\end{proof}

\Cref{lem:arbitrary_a} shows that edge weights need to be uniformly updated only across edges sharing a common head in $\phi_i$ or a common tail in $\psi_i$. Still, there is nothing gained or lost by updating all non-finite edge weights to the same value.

\begin{theorem}
    Let  $V=\{v_1,\ldots,v_m\}$ where $v_i\in (\R\cup \{\pm \infty\})^n$ for $i\in [m]$ subject to \cref{assumption:not_identically_infinite}. Any optimal solution to $\mathrm{MCF}$ applied over $\mathcal{G}^{-}$ yields an optimal solution to $\tfw$.
\end{theorem}
\begin{proof}
    Let $\rho^*$ denote an optimal solution to $\mathrm{MCF}$ applied over $\mathcal{G}^-$, which we assume exists. 
    Write $p_i$ as the vector whose $j$-th component is given by
    \[
    \frac{\rho^*_{ji'}}{w_i}-\frac{\rho^*_{ij}}{w_i}+(2\lambda_i-1).
    \]
    First, note that $\sum_j(p_i)_j=\alpha_i-\beta_i+n\delta_i=0$, which is necessary for $p_i\in B^\circ_{\lambda_i}$, and which follows from the fact that $\sum_j\rho^*_{ij}=\beta_i^w$ and $\sum_j\rho^*_{ji'}=\alpha_i^w$ by balance of flow constraints over the nodes $\phi_i$ and $\psi_i$ respectively. 
    By construction, $p_i\in \partial \gamma_{\mathrm{tr}}^{\lambda_i}(x)$ for any points in the elementary convex set defined by $\rho^*$. That $\rho^*$ defines such a (non-empty) region follows from strong duality and the equivalence between tropical sectors and the translated cones of the normal fan to the polar dual. 
    Indeed, for any such $x$, it is easy to verify that $\langle p_i,x \rangle=\gamma_{\mathrm{tr}}^{\lambda_i}(v_i-x)$. Alternatively, one could note that by optimality of $\rho^*$, any $j$ such that $\rho^*_{ij}>0$ must have $j\in M_i$ (where $M_i$ is defined as in \cref{lem:subdiff_defined}), and similarly, $j\in N_i$ for any $\rho^*_{ji'}>0$, and thus conclude that 
    \[
    p_i\in \mathrm{conv}(\{\alpha_i\ones_j-\beta_i\ones_{j'}+\delta_i\ones : j \in M_i, j' \in N_i\}).
    \]
    In any case, the important thing is that
    \[
    \biggl(\sum_iw_ip_i\biggr)_j=\sum_i \bigl (\rho^*_{ji'}-\rho^*_{ij} +w_i\delta_i\bigr)=-\delta^w_{tot}+\sum_i w_i\delta_i = 0,
    \]
    in every component $j\in [n]$, i.e., $\sum_iw_ip_i=\zeros$. By \cref{thm:Durier}, our elementary convex set $\mathcal{C}_p$ defined by $P=(p_i)_{i\in[m]}$ is optimal. Since $\sum_iw_ip_i\in\partial f(x)$ (by the definition in \cref{lem:subdiff_defined}), then any $x^*\in \mathcal{C}_p$ is an optimal solution to $\tfw$ according to \cref{def:fw_subdifferential}.
\end{proof}

It is important to note that in contrast with cases where $V\subset\TPT{n}$ (for which a finite Fermat--Weber solution always exists), non-finite data can result in $\mathcal{X}^*\cap \R^n=\emptyset$. This corresponds to situations where the min-cost flow problem over $\mathcal{G}^-$ is infeasible. One can determine feasibility of the min-cost flow problem by solving a max flow ($\mathrm{MF}$) problem over $\mathcal{G}^{-}$ in the standard manner, see \cite{AMO_NetworkFlows}.

Of course, one could consider simply replacing $\infty$ (resp. $-\infty$) elements of $V$ by sufficiently large (resp. small) real numbers and then solving $\mathrm{TFW}$ as in the finite case. If the original problem is feasible, then this method will return the correct solution. Solving the restricted problem over $\mathcal{G}^{-}$, however, identifies infeasibility without additional checks required, and besides, solving $\mathrm{MCF}$ over a sparser graph $\mathcal{G}^{-}$ is generally preferable to a denser one.

\section{Inverse Tropical Fermat--Weber Problems}\label{sec:inverse}

Given an optimization problem and an initial feasible solution $x^0$, the \emph{inverse problem} asks to find a new coefficient vector for the objective function such that $x^0$ is optimal in the modified problem. Inverse optimization problems have roots in the geophysics community \cite{Tarantola2005}, and appear to have gained attention in the 1990's with work on inverse linear programming and combinatorial optimization problems \cite{ZhangLiu1996, AhujaOrlin2001, Heuberger2004}. 

Inverse location problems over trees and general networks were studied in \cite{BermanIngcoOdoni1992,BermanIngcoOdoni1994} where the authors consider ``reductions'' (modifications that maintain the current network configuration) as well as ``additions'' (adding entirely new arcs to the network). In \cite{BurkardPleschiutschnigZhang2004}, the authors consider (inverses of) \emph{$p$-median} problems with both positive and negative weights. Allowing for negatively weighted sites (sometimes called \emph{obnoxious facilities}) generalizes the Fermat--Weber (resp. $p$-median) problem to a broader class of attraction/repulsion problems sometimes referred to as the \emph{extended Fermat--Weber problem}. Problems of this type were explored \cite{Plastria1991MajorityFermatWeber, Plastria1996}. See also \cite{ChurchDrezner2022} for a more recent survey. In this section, however, we will assume non-negative weights only.

The \emph{inverse tropical Fermat-Weber problem} (ITFW) can be stated as follows. For a tropical Fermat--Weber problem $\tfw(V,\Lambda,w)$, and feasible primal solution $x^0 \in \TPT{n}$, find weights $w'\in \R^m_{\geq0}$ that solve the following optimization problem:
\begin{equation*}
    \mathrm{ITFW}(V,\Lambda,w,x^0): \quad \underset{w'\in \R^m_{\geq0}}{\text{minimize}}\,\, f(w')\\
\quad \text{subject to} 
\quad x^0 \in \tfw(V,\Lambda,w')^*,
\end{equation*}
where $f(w')$ is a function that penalizes the degree to which one deviates from the original weight vector $w$. Typical examples include $f(w')=\lVert w'-w \rVert$ for some norm $\lVert \cdot \rVert$ (often $\ell^1$ or $\ell^\infty$) of the objective.

\subsection{Inverse Feasibility Problem}

It will first be helpful to consider some conditions under which feasible solutions to our inverse problem exist. For example, a completely trivial solution, which we do not exclude in $\mathrm{ITFW}$, can be given by $w'=\zeros$. That is, our problem is \emph{always} feasible in the strictest sense. This ends up being harmless, as any \emph{reasonable}\footnote{In particular, any norm.} penalty function $f(w')$ will avoid the trivial solution unless this is the \emph{only} feasible solution. Thus, we can regard an ``optimal'' trivial solution as an indication of infeasibility in the stricter sense. Thus, when we say ``feasible solution'' in the sequel, the reader can assume we mean $w'\neq \zeros$.

We begin by reviewing some known results. Consider the possibility $x^0\in \tilde{v_i}$ for some $v_i\in V$, that is, $x^0=v_i+c\ones$ for some $c\in\R$. Then any $w'$ such that 
\[
w_i\geq\sum_{j\in [m]\setminus \{i\}}w_j,
\]
provides a feasible solution. For weights $w'$ so applied, point $v_i$ is said to \emph{hold a majority} in $V$, see \cite{Plastria1991MajorityFermatWeber}. 

If our problem is fully asymmetric, i.e., $\Lambda=\zeros$ or $\Lambda=\ones$, then a point $x^0$ is feasible \emph{if and only if} $x^0\in \mathrm{tconv}(V)$ where $\mathrm{tconv}(V)$ is the \emph{tropical convex hull}\footnote{$\Lambda$ implies which version of tropical convex hull is appropriate. For $\Lambda=\zeros$, the min-plus tropical convex hull is used, and for $\Lambda=\ones$ the max-plus tropical convex hull is the appropriate one.} of $V$. 
\begin{definition}[Tropical Convex Hull, \cite{Joswig2021Essentials}]
    For a subset $V \subset \Rmin^n$, the min-plus \emph{tropical convex hull}
    \[
    \mathrm{tconv}_{\mathrm{min}}(V) \coloneq \{c_1 \odot v_1 \oplus \dots \oplus c_m \odot v_m \mid c_i \in \R, v_i \in V\}
    \]
    is the smallest tropically convex subset containing $V$. The set $\mathrm{tconv}(V)$ is called a \emph{tropical polytope} if it admits a finite generating set.
\end{definition}
That $x^0\in \mathrm{tconv}_{\mathrm{min}}(V)$ is sufficient for $x^0$ to be feasible in $\mathrm{ITFW}$ (for $\Lambda=\zeros$) is a result that first appeared in \cite{CoxCuriel2023}, who showed (Theorem 3.3) that for any covector cell $C\subseteq\mathrm{tconv}_{\mathrm{min}}(V)$ there exists a choice of real \emph{positive} weights $w'$ such that $C=\tfw(V,\Lambda,w')^*$. The proof in \cite{CoxCuriel2023} relies on a correspondence between subsets of vertices of a product of simplicies $\Delta^{m-1}\times\Delta^{n-1}$ and bipartite graphs, allowing the authors to provide an explicit formula for $w'\in \R^n_{>0}$. 

\begin{theorem}\label{thm:Minkowski_Weyl}
    For $x^0\in \TPT{n}$, $V=\{v_1,\ldots,v_m\}\subset \TPT{n}$, and $\Lambda=\zeros$, there exists a vector of positive weights $w'\in \R^m_{> 0}$ such that $x^0 \in \tfw(V,\Lambda,w')$ \emph{if and only if} $x^0 \in \mathrm{tconv_{min}}(V)$. Similarly, if $\Lambda=\ones$ then there exists a vector of non-negative weights $w'\in \R^m_{\geq 0}$ such that $x^0 \in \tfw(V,\Lambda,w')$ \emph{if and only if} $x^0 \in \mathrm{tconv_{max}}(V)$.
\end{theorem}
\begin{proof}
    Consider $\Lambda=\zeros$. Sufficiency is given by \cite[Theorem 3.3]{CoxCuriel2023}. Necessity follows from the fact that $\mathrm{tconv}_{\mathrm{min}}(V)$ is precisely the union of bounded cells of $\mathscr{C}(V,\zeros)$. It is easy to show $f(x)$ is coercive, and thus, solutions must be bounded. The case when $\Lambda=\ones$ follows symmetrically when considering $\mathrm{tconv}_{\mathrm{max}}(V)$.
\end{proof}

Tropical convex hull arguments used in the asymmetric setting do not generally extend to the case of mixed/hybrid gauges. Indeed, some tropical hyperplanes may be entirely missing from the decomposition given by $\mathscr{C}(V,\Lambda)$. Even so, since tropical convex hulls are exactly the support of bounded covector cells in the covector decomposition \cite{DevelinSturmfelsTropConvexity}, one might expect that by taking the union of bounded cell in $\mathscr{C}(V,\Lambda)$ we could arrive at a similar result. This is not the case, however, as can be shown even in the simple case of two points in $\TPT{3}$. \cref{fig:skew_gage_example} depicts the setup for several choices of bias, where the (closed) shaded region corresponds to the set of all possible tropical Fermat--Weber points given arbitrary non-negative initial weights. Intuitively, these sets also correspond to pairwise geodesics for paths
\begin{align*}
[u, v]_\lambda \coloneq & \argmin_x\{x\in \TPT{n} : \dtr^\lambda(u,x)+\dtr^\lambda(x,v)\}\\
 =& \argmin_x\{x\in \TPT{n} : \dtr^{1-\lambda}(x,u)+\dtr^\lambda(x,v)\}
\end{align*}
for tropical $\ell^{1,\infty}$ norms given specific values of $\lambda$ (as indicated below the figures). That such geodesic paths correspond to intersections of min-plus and max-plus tropical sectors was shown in \cite[Lemma 7]{Comaneci2024Convexity} specifically for $\lambda=0$. This is equivalent to a definition
\[
[u, v]_{\lambda=0} \coloneq \{ x\in \TPT{n} : x\in u+\sigma, v\in x+\sigma \},
\]
for some normal cone $\sigma\in B^\circ$ of the polar dual.
On both the far left ($\lambda=0$) and far right ($\lambda=1$) figures, neither the min-plus nor max-plus tropical convex hull contains the set of feasible solutions. Conversely, in the middle figure, $\lambda \in(0,1)$, there are four bounded cells for which no non-trivial solutions exist in the inverse problem.

Clearly, if $x^0\in[v_i,v_j]_\lambda$ for some pair $v_i,v_j\in V$, then setting $w'=c\ones_{\{i,j\}}$ for any $c>0$ results in $x^0\in \tfw(V,\Lambda,w')^*$. But there are also examples when $x^0$ is feasible with strictly positive $w'$ such that $x^0$ is not contained in any pairwise geodesic path (see \cref{fig:geodesic_counterexample}).

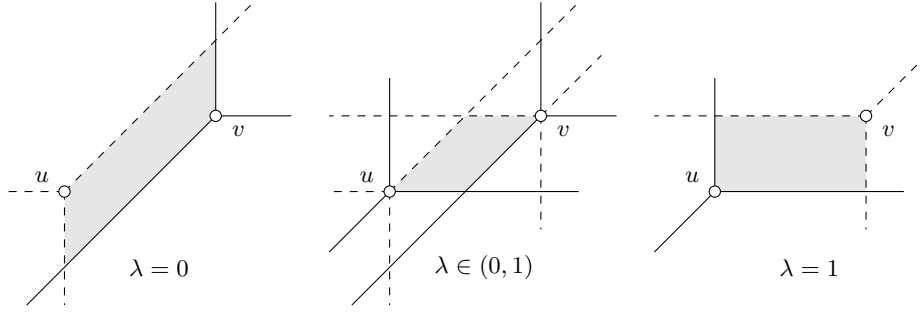
\begin{figure}
    \centering
    \begin{tikzpicture}%
    \fill[{gray!20}] (0,0) -- (1,1) -- (2,1) -- (1,0);
    \draw [-,](0,0) -- (0,1.5);
    \draw [-,](0,0) -- (2.5,0);
    \draw [-,](0,0) -- (-0.8,-0.8);
    \draw [-,dashed](0,0) -- (0,-1.5);
    \draw [-,dashed](0,0) -- (-0.8,0);
    \draw [-,dashed](0,0) -- (2.5,2.5);
    \draw [-,](2,1) -- (2,2.5);
    \draw [-,](2,1) -- (3,1);
    \draw [-,](2,1) -- (-0.5,-1.5);
    \draw [-,dashed](2,1) -- (2,-0.5);
    \draw [-,dashed](2,1) -- (-0.8,1);
    \draw [-,dashed](2,1) -- (2.8,1.8);
    \node at (0,0)[circle,draw=black,fill=white,inner sep=1.5pt]{};
    \node at (2,1)[circle,draw=black,fill=white,inner sep=1.5pt]{};
    \node at (-0.3,0.2) {\small$u$};
    \node at (2.3,0.8) {\small$v$};
    \node at (1.25,-1) {\small$\lambda\in(0,1)$};
    \begin{scope}[xshift=4.3cm]
        \fill[{gray!20}] (0,0) -- (0,1) -- (2,1) -- (2,0);
        \draw [-,](0,0) -- (0,1.5);
        \draw [-,](0,0) -- (2.5,0);
        \draw [-,](0,0) -- (-0.8,-0.8);
        \draw [-,dashed](2,1) -- (2,-0.5);
        \draw [-,dashed](2,1) -- (-0.8,1);
        \draw [-,dashed](2,1) -- (2.8,1.8);
        \node at (0,0)[circle,draw=black,fill=white,inner sep=1.5pt]{};
        \node at (2,1)[circle,draw=black,fill=white,inner sep=1.5pt]{};
        \node at (-0.3,0.2) {\small$u$};
        \node at (2.3,0.8) {\small$v$};
        \node at (1.25,-1) {\small$\lambda=1$};
    \end{scope}
    \begin{scope}[xshift=-4.3cm]
        \fill[{gray!20}] (0,0) -- (0,-1) -- (2,1) -- (2,2);
        \draw [-,dashed](0,0) -- (0,-1.5);
        \draw [-,dashed](0,0) -- (-0.8,0);
        \draw [-,dashed](0,0) -- (2.5,2.5);
        \draw [-,](2,1) -- (2,2.5);
        \draw [-,](2,1) -- (3,1);
        \draw [-,](2,1) -- (-0.5,-1.5);
        \node at (0,0)[circle,draw=black,fill=white,inner sep=1.5pt]{};
        \node at (2,1)[circle,draw=black,fill=white,inner sep=1.5pt]{};
        \node at (-0.3,0.2) {\small$u$};
        \node at (2.3,0.8) {\small$v$};
        \node at (1.25,-1) {\small$\lambda=0$};
    \end{scope}
    
    \end{tikzpicture}
    \caption{Sets of geodesic paths $[u,v]_\lambda$ for various $\lambda$. Tropical min-plus hyperplanes are drawn with solid rays, and max-plus hyperplanes with dashed rays. Note that for $V=\{u,v\}$ and $\Lambda=(1-\lambda,\lambda)$, the hyperplane arrangements depict $\mathscr{C}(V,\Lambda)$ for an associated (unweighted) tropical Fermat--Weber problem.}
    \label{fig:skew_gage_example}
\end{figure}

\begin{remark}
    In \cite{Comaneci2024Convexity}, a general setup according to 
    \[
    h(x)=g\bigl(f_1(x),\ldots,f_m(x) \bigr)
    \]
    where $h:\TPT{n}\rightarrow\R$ is the objective function, $g:\R^m_{\geq 0}\rightarrow\R$ is a scalarization of the multi-objectives $f_i$, and $f_i(x)=\overline{\gamma}_i(x-v_i)$ is a gauge $\gamma_i$ that can vary over $i\in [m]$. In particular, the author claims (Theorem 18) that if $g(x)$ is strictly increasing, and at least one $f_i(x)$ is strictly $\triangle$-star-quasiconvex, then all minima of $h(x)$ are contained in $\mathrm{tconv}_{\max}(v_1,\ldots,v_m)$. If this were true, then since $g(x)=\sum_iw_if_i(x)$ for $w_i>0$ is a strictly increasing function on $\R^n_{\geq0}$, then \cref{thm:Minkowski_Weyl} above could extend to any tropical Fermat--Weber problem containing at least one $\lambda_i\neq1/2$, since the associated gauge is strictly $\triangle$-star-quasiconvex by \cref{lem:star_quasiconvexity}. Unfortunately, this is not the case, as \cref{ex:counter_example} shows.
\end{remark}

\begin{example}\label{ex:counter_example}
    Take $P=(p_1,p_2,p_3)$ with $p_1=(-1,2,-1)$ and $p_2=p_3=(1,-2,1)$. Note that these choices define a region given by $\mathcal{C}_P=\bigcap_{i=1}^3 v_i+\sigma^\circ(p_i)$, which by \cref{eq:polar_cone} corresponds to 
    \[
    \mathcal{C}_P=S_2^{\max}(-v_1)\cap S_2^{\min}(-v_2)\cap S_2^{\min}(-v_3).
    \]
    Since $\sum_iw_i'p_i=0$, then \cref{thm:Durier} implies $\mathcal{C}_P=\mathcal{X}^*$ is optimal so long as the intersection is non-empty. Thus, $\mathcal{X}^*\nsubseteq \mathrm{tconv}_{\max}(V)$. Conversely, take $x^0\in \mathrm{relint}(C)$ as a point in the relative interior of the tropical convex hull. Note, the cell $C$ corresponds to the family $P$ as before, but with update $p_2= (1,1,-2)$. Noting that $p_3=-p_1$, it can easily be shown that $\sum_iw_i'p_i=0$ \emph{if and only if} $w'=(t,0,t)$ for $t\geq0$. Thus, $x^0$ is \emph{never feasible} unless we allow zero weights for $w'$. 
    \begin{figure}
    \centering
    \begin{tikzpicture}%
    \fill[{gray!20}] (0,1) -- (2,3) -- (2,2) -- (0,0);
    \fill[{red!20}] (1,1) -- (2,2) -- (2,1);
    \draw [-,dashed](0,1) -- (0,-0.3);
    \draw [-,dashed](0,1) -- (-0.8,1);
    \draw [-,dashed](0,1) -- (2.5,3.5);
    
    \draw [-,dashed](2,2) -- (2,2.5);
    \draw [-,dashed](2,2) -- (3,2);
    \draw [-,dashed](2,2) -- (-0.3,-0.3);

    \draw [-,dashed](2,0) -- (2,3.5);
    \draw [-,dashed](2,0) -- (3,0);
    \draw [-,dashed](2,0) -- (1.6,-0.4);

    \draw [-,red,thick](0,1) -- (1,1) -- (2,2);
    \draw [-,red,thick](0,1) -- (2,1) -- (2,0);
    \draw [-,red,thick](2,2) -- (2,0);
    
    \node at (0,1)[circle,draw=black,fill=white,inner sep=1.5pt]{};
    \node at (2,2)[circle,draw=black,fill=white,inner sep=1.5pt]{};
    \node at (2,0)[circle,draw=black,fill=white,inner sep=1.5pt]{};

    \node at (-0.2,1.2) {\small$v_2$};
    \node at (2.3,1.8) {\small$v_1$};
    \node at (2.3,-0.2) {\small$v_3$};
    \node at (1.2,1.7) {\small$\mathcal{X}^*$};
    \node at (1.7,1.3) {\small$C$};
    \node at (0.5,3) {\small$\Lambda=(0,1,0)$};
    \begin{scope}[xshift=-5cm]
        \fill[{gray!20}, draw=black, thick] (0,1) -- (1,2) -- (2,2) -- (1,1);
        \fill[{gray!20}, draw=black, thick] (2,0) -- (0,0) -- (0,1) -- (2,1);
        \draw [-,dashed](0,1) -- (0,2.2);
        \draw [-,dashed](2,1) -- (3.3,1);
        \draw [-,dashed](0,1) -- (-0.8,0.2);
        \draw [-,dashed](0,1) -- (0,-0.3);
        \draw [-,dashed](0,1) -- (-0.8,1);
        \draw [-,dashed](1,2) -- (2.5,3.5);
        
        \draw [-,dashed](2,2) -- (2,3.5);
        \draw [-,dashed](2,2) -- (3,2);
        \draw [-,dashed](2,2) -- (-0.3,-0.3);
        \draw [-,dashed](2,2) -- (-0.8,2);
        \draw [-,dashed](2,2) -- (2.8,2.8);
    
        \draw [-,dashed](2,0) -- (3,0);
        \draw [-,dashed](2,0) -- (1.6,-0.4);
        \draw [-,dashed](2,0) -- (2,-0.5);
        \draw [-,dashed](2,0) -- (-0.5,0);
        \draw [-,dashed](2,0) -- (3.3,1.5);
        \draw [-, thick](2,0) -- (2,2);
    
        
        \node at (0,1)[circle,draw=black,fill=white,inner sep=1.5pt]{};
        \node at (2,2)[circle,draw=black,fill=white,inner sep=1.5pt]{};
        \node at (2,0)[circle,draw=black,fill=white,inner sep=1.5pt]{};
    
        \node at (-0.2,1.2) {\small$v_2$};
        \node at (2.3,1.8) {\small$v_1$};
        \node at (2.3,-0.2) {\small$v_3$};
        \node at (1.7,1.3) {\small$\mathcal{X}^*$};
        
        \node at (0.5,3) {\small$\Lambda=(\tfrac{1}{2},\tfrac{1}{2},\tfrac{1}{2})$};
    \end{scope}
    \end{tikzpicture}
    \caption{Sample $V$ as given in \cref{ex:running_example} for $a,b,c=0$. On the left, tropical hyperplanes forming $\mathscr{C}(V,\Lambda)$ associated with $\Lambda=\tfrac{1}{2}\ones$ are drawn as dashed rays. Each pairwise geodesic is shaded gray. $\mathcal{X}^*$ is not contained in any pairwise geodesic, but is clearly optimal for any equally weighted $w'$. On the right, $\Lambda=(0,1,0)$ along with the corresponding decomposition $\mathscr{C}(V,\Lambda)$. The max-plus tropical convex hull $\mathrm{tconv}_{\max}(V)$ is drawn in red with solid line segments. The Fermat--Weber set for $w'=(t-1,t,1)$ given any $t>1$ is denoted by $\mathcal{X}^*$. Clearly, any $x^0\in \mathcal{X}^*$ is inverse feasible, and many such choices such that $x^0 \notin \mathrm{tconv}_{\max}(V)$ are possible. Additionally, any $x^0\in \mathrm{relint}(C)$ is infeasible in the inverse problem unless $w'$ is permitted to take values of $0$ over some indices.}
    \label{fig:geodesic_counterexample}
    \end{figure}
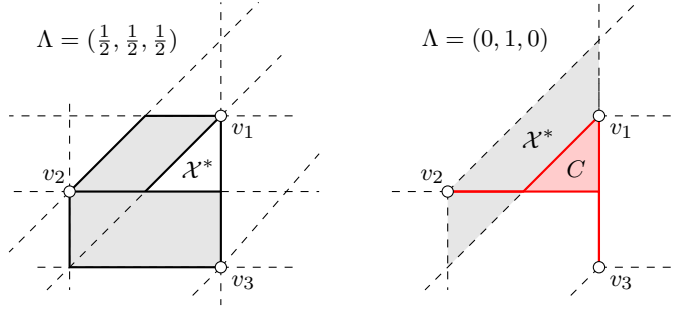
\end{example}


Our previous examples show that bounded cells of $\mathscr{C}(V,\Lambda)$ may not correspond to any pairwise geodesic path (\cref{fig:geodesic_counterexample} left), nor to a subset of a tropical convex hull (\cref{fig:geodesic_counterexample} right). The next lemma provides a suitable characterization in terms of tropical sector intersections.
\begin{lemma}\label{lem:bounded_intersection}
    Let $\mathscr{C}(V,\Lambda)$ denote cellular decomposition of $\R^n$ induced by union of min-plus and max-plus hyperplane arrangements as given in \cref{eq:cellular_decomposition}. Then for arbitrary $x\in \R^n$, $x$ is contained for a \emph{bounded} cell $C$ of $\mathscr{C}$ \emph{if and only if} there exist $v_i,\,v_k\in V$ such that $\lambda_i<1$, $\lambda_k>0$, and $x\in \{S^{\max}_j(-v_i)\cap S^{\min}_j(-v_k)\}$ for some $j\in[n]$.
\end{lemma}
\begin{proof}
    The requirement that $\lambda_i<1$ is obvious, since otherwise $S^{\max}_j(\zeros)$ doesn't correspond to any cone $\sigma\in \mathcal{N}_i$. The same logic applies to $\lambda_k>0$. Thus, in what follows, when we write $S^{\max}_j(-v_i)$, it is implied that $\lambda_i<1$ also.
    
    Let $C=\{x : x\in S^{\max}_j(-u), x\in S^{\min}_j(-v)\}$ for some $u,v\in V$ and $j\in[n]$. We first show that $C$ is bounded. By definition of tropical sectors, $x_i-u_i\leq x_j-u_j$ for $i\in[n]$ and any $x\in C$. Similarly, $x_j-v_j\leq x_k-v_k$ for $k\in [n]$. Together, these imply
    \[
    x_i-x_k\leq u_i-u_j+v_j-v_k=(u_i-u_k)+\underbrace{(v_j-u_j-v_k+u_k)}_{(*)}.
    \]
    The term with $(*)$ on the right hand side is bounded from above by $\dtr(u,v)$. Since $x_i-x_k=(x_i-x_j)+(x_j-x_k)$, we see that every pairwise difference must similarly be bounded. Thus, the coordinate differences in $\TPT{n}\simeq \R^{n-1}$ are bounded. Since $x \in C$ was arbitrary, $C$ must be bounded.
    
    Now let $C$ be an arbitrary bounded cell of $\mathscr{C}(V,\Lambda)$. We must show that for arbitrary $x\in C$ there exist $u,v\in V$ such that $x\in\{S^{\max}_j(-u)\cap S^{\min}_j(-v)\}$ for some $j\in[n]$. Suppose $\#V=m$, and let $A=(A_i \mid A_i\subseteq [n])$ for $i\in [m]$ (resp. $B=(B_k \mid B_k\subseteq [n])$ for $k\in [m]$) denote the (possibly empty) indices of the max-plus (resp. min-plus) sectors containing $C$. If we can find $j\in[n]$ such that $j\in A_i$ and $j\in B_k$, then we are done. Suppose then that no such $j$ exists. Let $M=\cup_iA_i$ denote the union of all indices in $A$, and similarly for $N=\cup_kB_k$. By our earlier assumption $M\cap N=\emptyset$. Let $d=\ones_M-\ones_N$ and consider $x'=x+td$ for some $t>0$. All inequalities implied by $A$ and $B$ remain valid for $x'$, even for arbitrarily large $t$, which implies $C$ is unbounded. Thus, we arrive at a contradiction, meaning there must exist $j\in[n]$ with $j\in A_i$ and $j\in B_k$. Take $i,k$ as the indices for the two points in $V$, i.e., $u=v_i$ and $v=v_k$, so that $x\in \{S^{\max}_j(-u)\cap S^{\min}_j(-v)\}$.
\end{proof}

If we allow ourselves to modify the gauges by adjusting $\lambda'_i$ in addition to $w_i'$, then clearly the problem becomes much easier. In particular, $x^0\in \mathrm{tconv}_{\mathrm{min}}(V)$ becomes sufficient by setting $\lambda'\leftarrow \zeros$, and similarly for $x^0\in[v_i,v_k]_\lambda$ by setting $\lambda_i\leftarrow (1-\lambda)$, $\lambda_k\leftarrow \lambda$, and $w'\leftarrow \ones_{\{i,j\}}$ as outlined earlier. This most relaxed version allows for the following characterization of the region associated to feasible $x^0$ in the inverse problem.

\begin{theorem}
    For any $x^0\in C$, where $C$ is a closed, bounded cell of $\mathscr{C}(V,\Lambda)$, there exist weights $w'$ and biases $\Lambda'$ such that $x^0\in \tfw(V,w',\Lambda')$.
\end{theorem}
\begin{proof}
    Since $x^0\in C$ for bounded and closed $C\in \mathscr{C}(V,\Lambda)$, then by \cref{lem:bounded_intersection} there exist $v_i,v_k\in V$ and $j\in[n]$ such that $x^0\in \{S^{\max}_j(-v_i)\cap S^{\min}_j(-v_k)\}$. Setting $\Lambda'\leftarrow \ones_i$ means that $x^0\in [v_i,v_k]_{\lambda=0}$. Setting $w'\leftarrow \ones_{\{i,k\}}$ means that every other point is trivial. Thus, $x^0\in \tfw(V,\Lambda',w')$. 
\end{proof}

Adjusting $\Lambda'$ to or from the boundary of the interval $[0,1]$ may not be realistic or desirable in many instances. Such prohibitions may correspond to gauges whose bias is fixed. A less intrusive modification for $\Lambda'$ might require, for instance, disallowing any additions/deletions of tropical hyperplanes from $\mathscr{C}(V,\Lambda)$. This is equivalent to requiring that we maintain the original structure on $\mathcal{G}$, the graph associated to the dual min-cost flow problem. In the notation of \Cref{sec:preliminaries}, this would equate to solving for $\alpha,\,\beta \in \R^m_{\geq 0}$ subject to $\lambda_i=0$ if $\alpha_i=0$ and similarly, $\gamma_i=1$ if $\beta_i=0$.  We call this the \emph{relaxed} inverse Fermat--Weber feasibility problem ($\mathrm{RITFW}$). 

The following lemma shows that checking for feasibility on this relaxed problem is equivalent to checking on the original problem. That is, if $x^0$ is feasible in $\mathrm{RITFW}(V,\Lambda,w, x^0)$, then there exist non-negative weights $w'$ such that $x^0\in \tfw(V,\Lambda,w')$.

\begin{lemma}\label{lem:x0_inverse_feasible}
    Let $V=\{v_1,\ldots,v_m\}\subset \TPT{n}$ and $\Lambda\in [0,1]^m$. Suppose we have found weights $\alpha,\,\beta\in \R^m_{\geq0}$ such that $x^0\in \mathcal{X}^*$ is optimal and so that $\alpha_i=0$ for all $i\in \{j\mid \lambda_j=0\}$ and $\beta_i=0$ for all $i\in\{j : \lambda_j=1 \}$. Then the problem $\mathrm{ITFW}(V,\Lambda,w,x^0)$ is feasible.
\end{lemma}
\begin{proof}
    We must show that there exists $w\in\R^m_{\geq0}$ such that $\alpha_i=w_i(1-\lambda_i)n$ and $\beta_i=w_i\lambda_in$. Consider arbitrary $i\in[m]$. There are three cases: $\lambda_i=0$, $\lambda_i=1$, and $0<\lambda_i<1$. If $\lambda_i=0$, then set $w_i=\alpha_i/n$. If $\lambda_i=1$, set $w_i=\beta_i/n$. If $0<\lambda_i < 1$, we need $\lambda_i=1-\alpha_i/(w_in)=\beta_i/(w_in)$. Thus, set $w_i=(\alpha_i+\beta_i)/n$ in this final case. By construction, $w$ is non-negative, so we are done.
\end{proof}

\Cref{lem:x0_inverse_feasible} says that the feasibility question depends only on the support of the weights ($\alpha,\beta$) rather than the their actual values, which is precisely what determines which hyperplanes are included in $\mathscr{C}(V,\Lambda)$. In the relaxed inverse problem, we allow these supports to vary, while in the original formulation they must remain fixed.

\subsection{Inverse Primal Problem}

As mentioned earlier, when a feasible inverse solution exists it is generally not unique, and one typically wishes to find such a feasible solution that minimizes some cost function associated with some measure of deviation from the original weight vector. In \cite{ZhangLiu1996} (and later in \cite{AhujaOrlin2001}), the authors show that linear programming problems of the form $\mathrm{(P)}$ with $f(w')=\lVert w'-w\rVert$ under an $\ell^1$ or $\ell^\infty$ norm have inverses that can be formulated as linear programs. For example, one writes $w'=w+a-b$ where $a,b\in \R^m_{\geq 0}$ and minimizes $f(w')=\langle c_a,a\rangle + \langle c_b,b\rangle$ where $c_a,c_b\in\R^m_{\geq0}$ are fixed non-negative costs associated with increasing or decreasing $w$. When $c_a=c_b=\ones$, this simplifies to the $\ell^1$ minimization case. 

We now provide linear programming formulations for the $\mathrm{ITFW}$ problem in the $\ell^1$ and $\ell^\infty$ case. These formulations largely follow those in \cite{AhujaOrlin2001}, with a minor specialization that accounts for the additional constraints required for net supply to equal net demand in the associated network flow problem. These additional constraints essentially enforce
\[
\delta'_{tot}=\sum_iw'_i(2\lambda_i-1)
\]
Let $\mathcal{A}^0$ denote the arc set of the equality subgraph at $x^0$, i.e., the arcs corresponding to those constraints in $(\mathrm{{P}})$ that are tight at $x^0$.
Then we can solve the primal inverse tropical Fermat--Weber problem using the following LP (cf. \cite{AhujaOrlin2001}, Section 4):
\begin{equation*}
    \begin{aligned}
        \underset{\rho,\,a,\,b\,\geq 0}{\text{minimize}} & \quad \langle c_a,a\rangle + \langle c_b, b\rangle \\
        \text{s.t.}\quad & \sum_{\{j:(i,j)\in \mathcal{A}^0\}}-\rho_{ij} + \beta_i(a_i-b_i)=-w_i\beta_i \quad &i \in [m],\\
        & \sum_{\{j:(j,i')\in \mathcal{A}^0\}}\rho_{ji'} - \alpha_i(a_i - b_i)=w_i\alpha_i \quad &i \in [m],\\
        & \sum_{\{i:(i,j)\in \mathcal{A}^0\}}\rho_{ij} - \sum_{\{i:(j,i')\in \mathcal{A}^0\}}\rho_{ji'} - \sum_i\delta_i(a_i-b_i)=\delta_{tot} \quad &j \in [n],
    \end{aligned}
\end{equation*}
where we emphasize that $\alpha_i$, $\beta_i$, and $\delta_i$ are the \emph{non-weighted} version of the parameters. This LP has at most $2m+n$ constraints in at most $m(n+2)$ variables.
Formulating the $\ell^\infty$ minimization problem is straightforward extension of the above LP. Introduce one additional variable $\theta$, which we intend to minimize, subject to additional constraints
\[
\theta\geq (c_a)_ia_i + (c_b)_ib_i \quad \text{for }i\in[m].
\]

It is instructive to note that the formulation above only depends on $V$ or $x^0$ in determining the arcs of the equality subgraph $\mathcal{A}^0$. This reflects the fact that it is the directional information of the subdifferentials at $x^0$ that matters once the actual distances become fixed. Finally, we see that while the original primal tropical Fermat--Weber problem possesses a totally unimodular constraint matrix (see \cref{rem:total_unimodularity}), the inverse problem does not. In the next subsection, we see that this loss of total unimodularity also occurs in the dual inverse problem.


\subsection{Inverse Dual Problem}

One may also consider the inverse of the dual problem. In this setting, the question becomes, ``Given some feasible flow to a min-cost flow formulation, which arc costs make this flow optimal?'' As with the primal inverse, the goal is typically to select these new arcs costs in such a way as to minimize some measure of perturbation from the original arc costs, typically $\ell^1$ or $\ell^\infty$. As in the primal inverse, the feasibility question can be considered vacuous unless we impose restrictions on the coordinates of $v_i\in V$. For example, a trivial solution is given by perturbing all points to lie in the same equivalence class (i.e. $\dtr(v_i,v_k)=0$ for all $v_i,v_k \in V$). 

Each feasible flow in $\mathrm{MCF}$ corresponds to a subset of constraints in $\mathrm{(P)}$ that are met with equality for some optimal point $x^*\in \mathcal{X}^*$. Since these tight constraints correspond to tropical sectors, the problem asks to perturb the apices of a tropical hyperplane arrangements so that the intersection of the fixed tropical sectors corresponding to $\rho^0$ becomes non-empty.

In the generic case, the authors in \cite{AhujaOrlin2001} show that the inverse min-cost flow problem under the $\ell^1$ norm is again a min-cost flow problem. This, however, assumes perturbation of arc costs can be done independently, an assumption that is violated any time we have $\lambda_i\in(0,1)$. Indeed, for any $\lambda_i\in(0,1)$, an arc $\rho_{ij}$ must have a per unit cost that is the additive complement of the per unit cost of an arc $\rho_{ji'}$. Coupling arc costs in this fashion means that, in general, the inverse tropical Fermat--Weber dual problem cannot be easily transformed into a min-cost flow problem as in \cite{AhujaOrlin2001}.

Let $\rho^0$ be any feasible flow solution to $(\mathrm{D})$. As depicted in \cref{fig:mcf_formulation}, arcs in the network flow graph are partitioned according to $\mathcal{A=\mathcal{A_{\max}}\sqcup \mathcal{A}_{\min}}$, and so we define
\[
\mathcal{A^+_{\max}}=\{e\in \mathcal{A}_{\max}: \rho^0_e>0\},\qquad 
\mathcal{A^+_{\min}}=\{e\in \mathcal{A}_{\min}: \rho^0_e>0\},\\
\]
as the subset of arcs in each partition with positive (binding) flow values. Similarly, let $\mathcal{A}^0_{\max}$ and $\mathcal{A}^0_{\min}$ denote the subset of arcs carrying zero flow in $\rho^0$. We again use $i'\in[m]$ as a notational device to indicate that $i$ and $i'$ are separate nodes whose arc costs are coupled by the relation $c_{ij}=-c_{ji'}$.

The inverse dual tropical Fermat--Weber problem ($\mathrm{IDTFW}$) asks to find $U\in \R^{m \times n}$ at minimal cost such that $\rho^0$ is optimal in $(\mathrm{D})$. Here, $U_{ij}=V_{ij}+a_{ij}-b_{ij}$ and \emph{minimal cost} is computed as a (not necessarily uniform or symmetric) penalty on the perturbations of $V$ as in the primal inverse formulation. Again, when $c_a=c_b=\ones$, this is equivalent to minimizing using the $\ell^1$ norm. As we did with the primal inverse, we again convert to the dual and work with node potentials $\pi=(\phi,x,\psi)$. Complementary slackness requires that
\begin{align*}
    x_j-\phi_i - a_{ij} + b_{ij} &= V_{ij} \quad &(i,j)\in \mathcal{A^+_{\max}},\\
    \phi_i - x_j + a_{ij} - b_{ij} &= -V_{ij} \quad &(j,i')\in \mathcal{A^+_{\min}},
\end{align*}
which shows $\phi_i=\psi_i$ any time both $\rho^0_{ij}>0$ and $\rho^0_{ji'}>0$. That $\phi_i\geq \psi_i$ holds in general follows directly from the dual constraints. Constraints associated to arcs in $\mathcal{A}^0_{\max}$ and $\mathcal{A}^0_{\min}$ are similar, but with inequalities $\leq$ in place of the equalities. Thus, (IDTFW) can be solved as an LP with at most $2mn$ constraints in at most $2m+n+2mn$ variables. 

A potentially useful alternative formulation comes from eliminating the $a_{ij}$ and $b_{ij}$ variables in order to work purely over the node potentials $\pi$. For simplicity in what follows, we assume minimization using the $\ell^1$ norm. First, note that
\[
f_e(\pi)=\lvert U_{ij}-V_{ij}\rvert = \begin{cases}
    \lvert x_j-\phi_i - V_{ij}\rvert \quad &\rho^0_{ij}>0,\\
    \lvert x_j-\psi_i - V_{ij}\rvert \quad &\rho^0_{ij}=0, \rho^0_{ji'}>0,\\
    d(V_{ij}, [x_j-\phi_i, x_j-\psi_i]) \quad &\rho^0_{ij}=\rho^0_{ji'}=0,
\end{cases}
\]
where
\[
d(V_{ij}, [x_j-\phi_i, x_j-\psi_i]) = \begin{cases}
    (x_j-\phi_i) -V_{ij} \quad &\text{if } V_{ij}< x_j-\phi_i,\\
    V_{ij} - (x_j-\psi_i) \quad &\text{if } V_{ij} > x_j-\psi_i,\\
    0 \quad &\text{otherwise,}
\end{cases}
\]
is the projection of $V_{ij}$ onto the interval defined by $[x_j-\phi_i,x_j-\psi_i]$. Our earlier observation $\phi_i\geq \psi_i$ implies that this interval is always nonempty. This shows that $f_e(\pi)$ is a hinge loss. The problem (IDTFW) can now be written as
\begin{equation*}
    \begin{aligned}
        \mathrm{IDTFW}(V,\rho^0): \quad \underset{\pi=(\phi,x,\psi)}{\text{minimize}} & \quad f(\pi)=\sum_ef_e(\pi) \\
        \text{s.t.}\quad & \psi_i \leq \phi_i \quad & i \in [m],\\
        & \phi_i \leq \psi_i  \quad & i \in \{i:\rho^0_{ij}>0, \rho^0_{ji'}>0\}.
    \end{aligned}
\end{equation*}
Subgradients for $f_e(\pi)$ (and thus for $f(\pi)$) correspond to signed characteristic vectors given by the indices of $\pi$ associated with $\phi_i,x_j,\psi_i$ respectively. For example, in the case where $\rho^0_{ij}>0$ and $x_j-\phi_i<V_{ij}$, we have $\nabla f_e(\pi)=\ones_u-\ones_v$ where $u,v$ denote the respective indices of $\phi_i,x_j$ in $\pi$.

This formulation implicitly assumed that nodes $\phi_i,\psi_i$ exist for all $i\in [m]$, which corresponds to the case where $\Lambda \in (0,1)^m$. This is without loss of generality, as $\rho^0_{ij}=0$ for any $\lambda_i=0$ and $\rho^0_{ji'}=0$ for any $\lambda_i=1$ must hold by feasibility of $\rho^0$. Thus, any associated variables $\phi_i,\psi_i$ will never require equality over their coupling, and their only influence on $f_e(\pi)$ is via the third case ($\rho^0_{ij}=\rho^0_{ji'}=0$) where they are seen to allow for a degeneration of the closed interval into a half-line without affecting feasibility ($\psi_i\leq \phi_i$). Thus, the cost associated with adjusting these variables can always be made to be zero, which is equivalent to the omission of the arc entirely.

The formulation above is not an LP, but it does allow for straightforward evaluations of $f(\pi)$ and $\nabla f(\pi)$, which suggests it may be useful for iterative approximate optimization techniques such as subgradient/bundle methods. Since initial cost data $V$ may in fact be noisy, it reasons that for certain applications, such an approximation of the objective function could prove satisfactory. Regardless of noise assumptions, mitigating the quadratic growth in the constraint matrix may be beneficial for larger problems.

\subsection{TFW Network Design Problem}

Our final formulation looks at a problem that is not an inverse programming problem in the sense that we do not modify the objective function coefficients. Nonetheless, it shares similarities to the problems previously described and may be interesting in its own right. Given $x^0\in \R^n$, suppose that rather than perturbing weights, as in $\mathrm{ITFW}$, we instead allowed one to perturb the sample points, as in $\mathrm{IDTFW}$. The objective is to find a minimum cost perturbation of the sample $U$ such that $x^0\in \tfw(U,\Lambda,w)^*$. Like in $\mathrm{IDTFW}$, assigning $u_i\leftarrow x^0$ for all $i\in [m]$ shows this problem is always feasible. However, unlike $\mathrm{IDTFW}$, the set of binding constraints will (in general) vary across perturbations. Unlike in $\mathrm{ITFW}$, where the values assumed by $\phi_i$ and $\psi_i$ are static, here they can vary. Thus, this problem is generally more difficult than the previous two encountered.

The problem we just described has an interpretation as a network design problem \cite{MagnantiWong1984}. Consider the min-cost flow graph $\mathcal{G}$, with arc set $\mathcal{A}$ (see \cref{fig:mcf_formulation}), for some arbitrary instance of $\mathrm{TFW}(V,\Lambda,w)$, and let $\mathcal{G}^0$ denote the corresponding network with all the arcs removed. Suppose that we wish to add arcs to $\mathcal{G}^0$ in a manner that achieves feasibility in $\mathrm{MCF}$, where each added arc $(i,j)$ incurs a fixed cost $c_{ij}$. That is, we seek (at minimal total cost) a set of arcs whose addition to $\mathcal{G}^0$ enables some feasible flow. This is a network design problem with zero flow costs, binary design variables $y$ (one for each arc in $\mathcal{G}$), and fixed design costs $c_{ij}$. 

Consider any sample point $v_i$, and assume $\lambda_i\in(0,1)$ so that all arcs $\{(i,j) : j\in[n]\}$ and $\{(j,i') : j\in[n]\}$ exist in $\mathcal{G}$. Adding arc $(i,j)$ to $\mathcal{G}^0$ implies $j\in \argmax (x^0-u_i)$, which necessitates a coordinate-wise perturbation of $u_i$ from $v_i$ of at least
\[
(x^0_{j^*}-V_{ij^*})-(x^0_j-V_{ij})\geq0, \quad j^*\in \argmax(x^0-v_i),
\]
in the $j$-th coordinate direction. Likewise, adding arc $(j,i')$ necessitates a coordinate-wise shift of at least
\[
(x^0_j-V_{ij}) - (x^0_{j^*}-V_{ij^*})\geq0, \quad j^*\in \argmin(x^0-v_i).
\]
Since perturbing $v_i$ beyond these thresholds provides no additional benefit, we define $c_{ij}\coloneq \phi_i-(x^0_j-V_{ij})$ for $\{(i,j) : \lambda_i>0\}$ and $c_{ji}\coloneq (x^0_j-V_{ij})-\psi_i$ for $\{(j,i) : \lambda_i<1\}$, where $\phi_i,\,\psi_i$ are defined as before. Note that defining costs in this way implicitly minimizes perturbation of $\lVert u_i-v_i\rVert_{\mathrm{tr}}$ under the tropical norm, rather than under $\ell^1$ or using the more general penalties considered previously. The \emph{network design} tropical Fermat--Weber problem $\mathrm{NDTFW}(V,\Lambda,x^0)$ can be written as
\begin{subequations}\label{eq:model}
    \begin{align}
    \underset{x,\,y}{\text{minimize}} 
    & \quad \sum_{(i,j)\in A} c_{ij} y_{ij} \tag{\theparentequation} \\
    \text{s.t.}\quad 
    & \sum_{(i,j)\in \mathcal{A}} x_{ij}
     - \sum_{(j,i)\in \mathcal{A}} x_{ji}
     = b_i
    && \text{for } i\in N,
    \\
    & 0 \leq x_{ij} \leq M y_{ij}
    && \text{for } (i,j)\in \mathcal{A}, \label{eq:model-Ky}
    \\
    & y_{ij}\in\{0,1\}
    && \text{for } (i,j)\in \mathcal{A},
    \end{align}
\end{subequations}
for ``big-$M$'', a sufficiently large positive number. Constraints \cref{eq:model-Ky} prevent flow on arcs that are not purchased. To recover $U$ from $y^*$, we consider two possibilities. First, if for any $i\in [m]$ there exists $j\in[n]$ such that we have $y_{ij}=y_{ji'}=1$, then $j\in\argmin(x^0-v'_i)$ and $j\in \argmax(x^0-v'_i)$ implies $\dtr(v_i,x^0)=0$. Indeed, checking the incurred arc costs yields $c_{ij}+c_{ji'}=\phi_i-(x^0_j-V_{ij})+(x^0_j-V_{ij})-\psi_i=\dtr(v_i,x^0)$. If no such $j\in [n]$ exists, then simply apply
\[
u_i=v_i+\sum_{(i,j)\in \mathcal{A}}c_{ij}y_{ij}\ones_j-\sum_{(j,i')\in \mathcal{A}}c_{ji'}y_{ji'}\ones_j,
\]
to obtain a representative of $u_i\in \TPT{n}$. 

In general, mixed-integer linear programs such as \cref{eq:model} can be solved using a branch-and-bound style algorithm. Alternatively, one could consider a Bender's decomposition method in which subproblems correspond to feasibility assessments on $\mathrm{MCF}$ given some proposed ``arc-purchase'' solution $y^k$. Since such subproblems can be computed efficiently, this latter approach seems promising.

\section{Discussion}\label{sec:Discussion}



We describe tropical $\ell^{1,\infty}$ (pseudo)-norms applied to tropical locations problems. By leveraging an equivalent optimality criterion given by conditions on the objective function subdifferential, we provide an extension to problems in which the data are not limited to finite coordinates. When finite optimal solutions to the extended problem exist, we show how simple modifications to a min-cost flow formulation enable efficient computation. 

We also consider inverse formulations of tropical Fermat--Weber primal and dual problems, providing some characterizations on the space of feasible inverse solutions and providing linear programming formulations to compute their solution. We show that, unlike in a general inverse min-cost flow problem, the inverse tropical Fermat--Weber problem requires specific symmetries across arc costs that result in an inverse formulation that is not in general totally unimodular.

As mentioned in \cref{sec:inverse}, Fermat--Weber problems (and their inverses) have been considered in the case of negative weights, and with both positive and negative weights. This latter extension leads to the class of Fermat--Weber attraction/repulsion problem, which can be highly non-convex. Difference of convex (DC) programming methods have been considered for these types of problems \cite{TuyAlKhayyalZhou1995,NickelDudenhoeffer1997}, and their utility for solving tropical attraction/repulsion problems offers an interesting avenue for future work.

\backmatter





\bmhead{Acknowledgments}

J.~S.~and R.~Y.~are partially supported by the NSF Division of Mathematical Sciences: Statistics Program DMS 2409819.

\bibliography{sn-bibliography}

\end{document}